\documentclass[11pt,reqno]{article} 

\usepackage[letterpaper,margin=1in]{geometry}
\usepackage{setspace}

\usepackage[T1]{fontenc}
\usepackage[utf8]{inputenc}
\usepackage{lmodern}

\usepackage{amsmath,amssymb,amsthm,mathtools}
\mathtoolsset{showonlyrefs} 
\numberwithin{equation}{section}

\usepackage{dsfont}

\newcommand{\K}{\mathbb{K}}

\newcommand{\1}{\mathds{1}}
\newcommand{\tW}{\mathtt{W}}
\newcommand{\tB}{\mathtt{B}}
\newcommand{\tV}{\mathtt{V}}
\newcommand{\tE}{\mathtt{E}}

\usepackage{thmtools}
\declaretheorem[name=Theorem]{theorem}
\declaretheorem[name=Conjecture]{conj}
\declaretheorem[name=Lemma]{lemma}
\declaretheorem[name=Proposition]{proposition}
\declaretheorem[name=Corollary]{corollary}
\declaretheorem[style=definition,name=Definition]{definition}
\declaretheorem[style=remark,name=Remark]{remark}

\usepackage{graphicx}
\usepackage{caption}
\usepackage{subcaption}
\usepackage{booktabs}
\usepackage{float}      
\usepackage{siunitx}   
\usepackage{url}
\usepackage{color}
\usepackage[dvipsnames]{xcolor}

\usepackage{tikz}
\usetikzlibrary{calc,arrows.meta,patterns}
\usepackage{pgfplots}
\pgfplotsset{compat=1.18}

\usepackage{enumitem}
\usepackage{microtype}

\newcommand{\ii}{\mathrm{i}}

\newcommand{\eps}{\varepsilon}

\newcommand{\dz}{\mathrm{d}z}

\newcommand{\dw}{\mathrm{d}w}

\let\phi\varphi

\usepackage{authblk} 
\title{The $t$-Split Two-Periodic Aztec Diamond Model}
\author[1]{Meredith Shea}
\affil[1]{Department of Mathematics and Statistics, Vassar College}
\date{\today}

\begin{document}
\maketitle
\begin{abstract}
    In this work we consider an Aztec diamond model split into two unequal regions which are asymptotically fixed in size. Each region is weighted with a distinct two-periodic weighting. We refer to this model as the $t$-split two-periodic Aztec diamond, to signify its difference from the previous work title \textit{Split Two-Periodic Aztec Diamond}, where the model was split into two equal regions. We derive an integral expression for the correlation kernel of the model and give a partial description of the scaling limit behavior, along with a conjecture for the remainder. We refer to the larger and smaller sides of the model as the \textit{dominant} and \textit{non-dominant} sides, and to the location of the weight change as the \textit{interface}. The dominant side exhibits a limit shape that depends only on its own weighting and is identical to that of the two-periodic Aztec diamond, while the non-dominant side appears to have a novel limit shape that depends on both weightings and the location of the interface. Lastly, we consider the complete limit shape in the case where the dominant side two-periodic parameter goes to 0. 
\end{abstract}

\section{Introduction}

\subsection{The Aztec Diamond and Dimer Models}

The \textit{Aztec diamond} of size $n$ is a subset of the square grid on $\mathbb{Z}^2$ which includes all squares whose centers, $(x,y)$, satisfy the inequality $|x| + |y| \leq n$. For more than a quarter of a century mathematicians have studied \textit{domino tilings} of the Aztec diamond, a problem first introduced in \cite{EKLP92}. A domino tiling of the Aztec diamond, or any subset of the square grid, is a placement of dominoes\footnote{Here a \textit{domino} is a $2 \times 1$ rectangle.} that covers the entire region and no dominoes overlap.

One can interpret domino tilings of the Aztec diamond as a stochastic process by assuming a domino tiling is chosen at random from the set of all possible tilings. The simplest case, which we call the uniform model, assumes that all domino tilings are equally likely. The scaling limit of this model, in which the size of the Aztec diamond, $n$, tends to infinity, was studied by Jockusch, Propp and Shor in \cite{JPS98}. They proved that the scaling limit of the model exhibits a deterministic region and that boundary between the deterministic and non-deterministic regions forms a circle. This has become known as the \textit{Arctic Circle Theorem}.

\begin{figure}[H]
    \centering
    \includegraphics[width=0.24\linewidth]{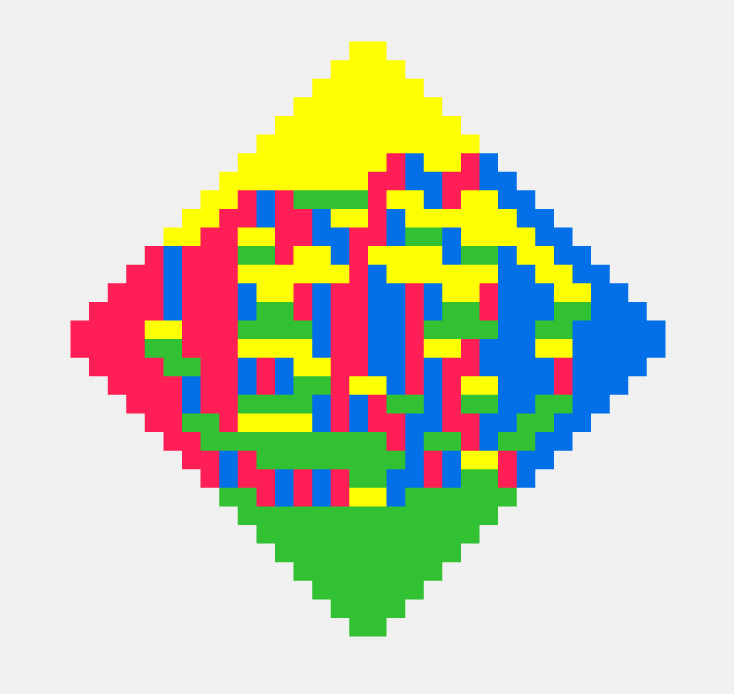}
    \includegraphics[width=0.24\linewidth]{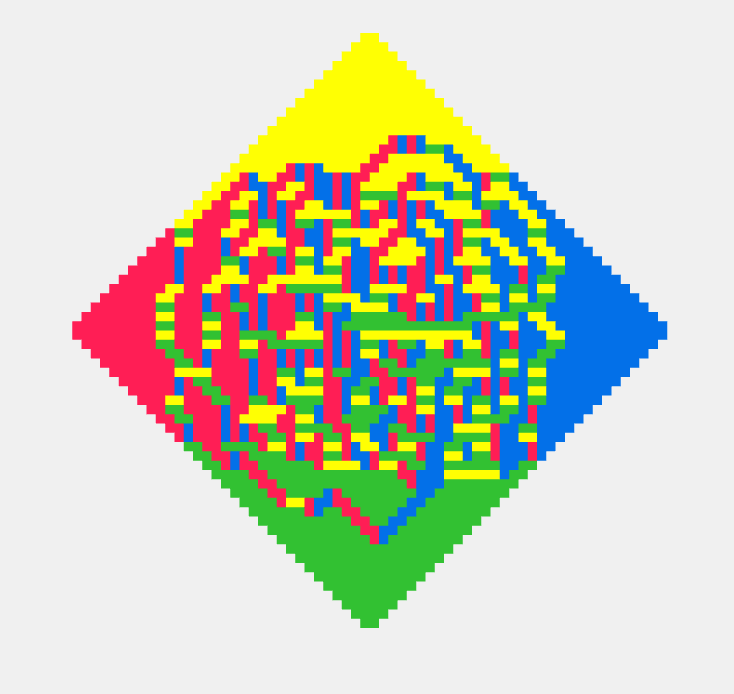}
    \includegraphics[width=0.24\linewidth]{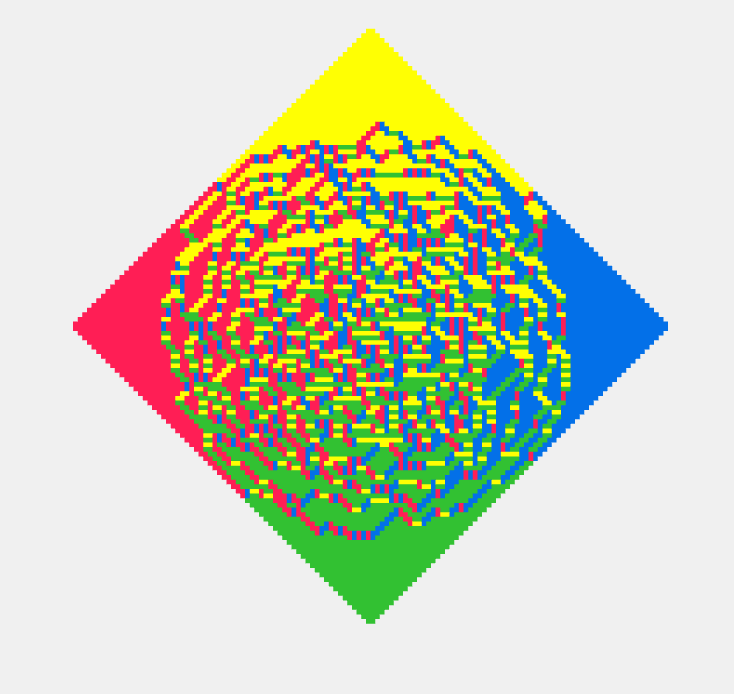}
    \includegraphics[width=0.24\linewidth]{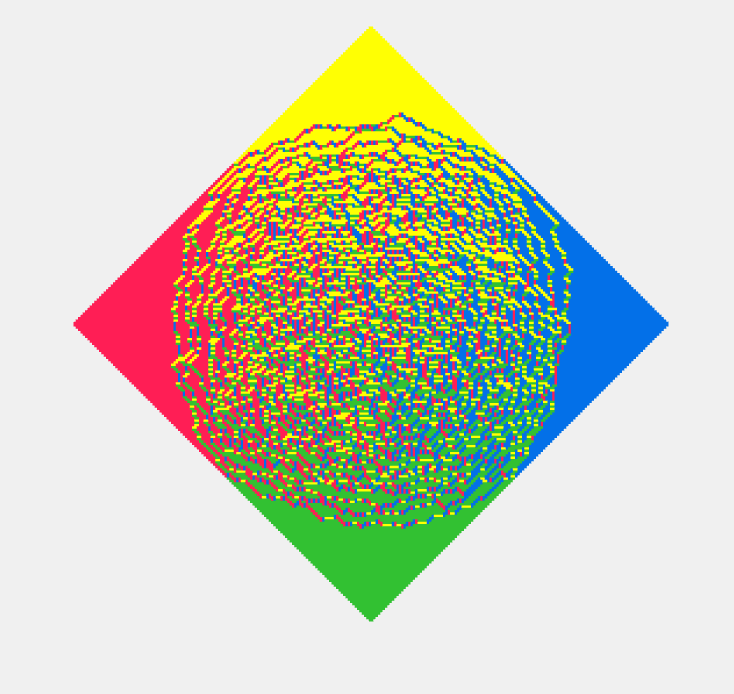}
    \caption{Simulations of the uniform Aztec diamond of sizes $n=16, \, 32, \, 64, \, 128$ to illustrate the Arctic Circle Theorem. Simulations were produced using the domino shuffling tool, \protect\url{https://lpetrov.cc/domino/}, created by Leonid Petrov.}
    \label{fig:actic-cicle-theorem}
\end{figure}

Domino tilings are a subset of a class of models known as \textit{dimer models}.\footnote{For a general primer on dimers see \cite{Ken2009lectures}.} A dimer model is defined on a graph and the possible configurations are known as \textit{dimer coverings}. A dimer covering is a subsets of the edges that match every vertex to exactly one neighbor. In graph theory, dimer coverings are also referred to as \textit{perfect matchings} of the graph. The ideas behind dimer models were initially introduced by Kasteleyn \cite{Kas61} and Temperley and Fisher \cite{TF61}. Associated to the dimer model is a natural Gibbs measure which can be interpreted in the following manner: to each edge of the graph we assign a weight (typically a positive, real number) and the probability of selecting any dimer covering is proportional to the product of the edge weights. The uniform tiling model is equivalent to all edge weights being equal. Kenyon, Okounkov, and Sheffield showed that there are three types of Gibbs measures that can appear in dimer model with doubly periodic edge weights \cite{KOS03}. These \textit{macroscopic regions} are classified as: frozen, rough, and smooth.\footnote{In the literature these regions are also referred to as solid, liquid, and gas, respectively.} In the frozen region dimers are deterministic. In the rough region dimer correlations decay polynomially with distance, while in the smooth region dimer correlations decay exponentially with distance. Not all dimer models exhibit smooth regions, however most exhibit rough regions.

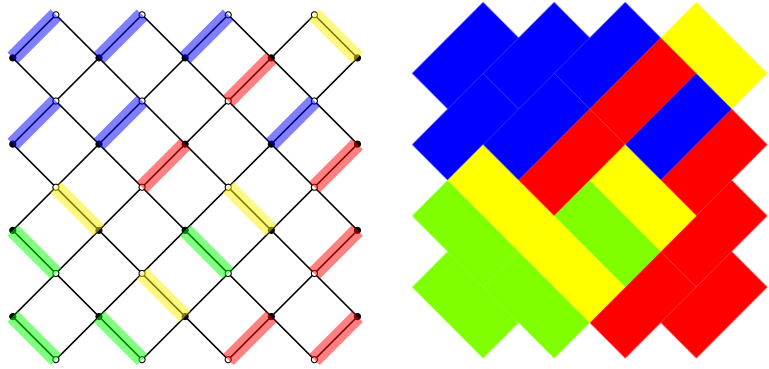
\begin{figure}[H]
    \centering
    \begin{tikzpicture}[scale=0.57]
    \foreach \t in {1,3,5,7}
        {
        \draw[semithick] (0,\t) -- (\t,0);
        \draw[semithick] (\t,8) -- (8,\t);
        \draw[semithick] (\t,0) -- (8,8-\t);
        \draw[semithick] (0,\t) -- (8-\t,8);
        }
    \foreach \j in {0,1,2,3,4}
        \foreach \k in {0,1,2,3}
            {
            \filldraw[black] (2*\j,2*\k+1) circle (2pt);
            }
    \foreach \j in {0,1,2,3}
        \foreach \k in {0,1,2,3,4}
            {
            \filldraw[color=black,fill=white] (2*\j+1,2*\k) circle (2pt);
            }
    \draw[line width=0.2cm,blue,opacity=0.5] (0,7) -- (1,8);
    \draw[line width=0.2cm,blue,opacity=0.5] (0,5) -- (1,6);
    \draw[line width=0.2cm,blue,opacity=0.5] (2,7) -- (3,8);
    \draw[line width=0.2cm,blue,opacity=0.5] (2,5) -- (3,6);
    \draw[line width=0.2cm,blue,opacity=0.5] (4,7) -- (5,8);
    \draw[line width=0.2cm,blue,opacity=0.5] (6,5) -- (7,6);
    \draw[line width=0.2cm,red,opacity=0.5] (7,0) -- (8,1);
    \draw[line width=0.2cm,red,opacity=0.5] (5,0) -- (6,1);
    \draw[line width=0.2cm,red,opacity=0.5] (3,4) -- (4,5);
    \draw[line width=0.2cm,red,opacity=0.5] (5,6) -- (6,7);
    \draw[line width=0.2cm,red,opacity=0.5] (7,4) -- (8,5);
    \draw[line width=0.2cm,red,opacity=0.5] (7,2) -- (8,3);
    \draw[line width=0.2cm,green,opacity=0.5] (0,1) -- (1,0);
    \draw[line width=0.2cm,green,opacity=0.5] (2,1) -- (3,0);
    \draw[line width=0.2cm,green,opacity=0.5] (4,3) -- (5,2);
    \draw[line width=0.2cm,green,opacity=0.5] (0,3) -- (1,2);
    \draw[line width=0.2cm,yellow,opacity=0.5] (7,8) -- (8,7);
    \draw[line width=0.2cm,yellow,opacity=0.5] (3,2) -- (4,1);
    \draw[line width=0.2cm,yellow,opacity=0.5] (1,4) -- (2,3);
    \draw[line width=0.2cm,yellow,opacity=0.5] (5,4) -- (6,3);
    \end{tikzpicture}
    \hspace{0.25cm}
    \includegraphics[scale=0.5,angle=90]{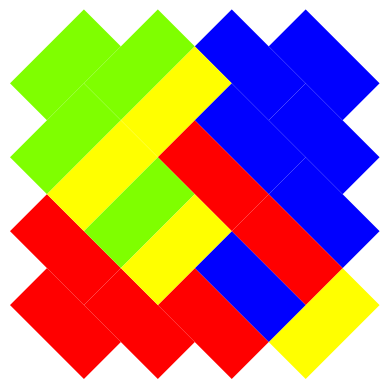}
    \caption{A dimer covering of the size 4 Aztec diamond and the related tiling picture.}
    \label{fig:AztecDiamondExamples}
\end{figure}

A key aspect towards understanding correlations of the Aztec diamond is the fact that the model is a \textit{determinantal point process}. One approach to computing correlations is via \textit{Kasteleyn theory}. The \textit{inverse Kasteleyn matrix} for the Aztec diamond with two-periodic weights was originally computed by Chhita and Young in \cite{CY14}. Further asymptotics of this model have been studied in \cite{CJY15,CJ16,BCJ18} among others. The two-periodic Aztec diamond is, notably, the simplest Aztec diamond model which exhibits all three macroscopic regions. 

Over the past decade there have been many advancement towards understanding general weights on the Aztec diamond. The correlation kernel for general periodic weightings was derived in the recent work of Berggren and Borodin \cite{BB23}. This analysis included asymptotics of the model. Additionally in \cite{BdT24}, Boutillier and de Tili\`{e}re compute the inverse Kasteleyn for Fock's weights on the Aztec diamond. This is the first result towards understanding non-periodic models, as Fock's dimer model is only quasi-periodic. 

This work is a follow up to \cite{S25} where an Aztec diamond dimer model with a non-periodic weighting scheme was analyzed. In particular, the \textit{split} two-periodic Aztec diamond divides the Aztec diamond into two equal regions and each region is assigned a different two-periodic weighting. The line where the change of weight occurs is referred to as the \textit{interface} of the model. In \cite{S25} a correlation kernel for the model was derived by extending methods defined in \cite{BD19}. The work also detailed the macroscopic regions of the model in the scaling limit. 

In the leading order, the local asymptotics of the split two-periodic model are determined only by the two-periodic weighting on the side of the interface where the coordinate lies. This means that the boundary between the macroscopic regions can be determine by computing the boundary in the separate two-periodic cases and then using the interface to ``glue'' them together.\footnote{In general, the two boundary curves will not meet up along the interface, however there will possibly be segments where a smooth region meets with a rough region at the interface.} Example simulations of the split two-periodic model are given in Figure \ref{fig:split-examples}. The earlier referenced work goes on to detail the regions of the split two-periodic model where the second order behavior changes from the typical two-periodic case. In particular, regions close to the \textit{half-cusps} that lie along the interface may have a different second order term. This poses interesting questions about the local behavior near these regions of the boundary.\footnote{However, these questions have yet to be addressed.}

\begin{figure}[H]
    \centering
    \begin{subfigure}[t]{0.3\textwidth}
        \includegraphics[width=\textwidth]{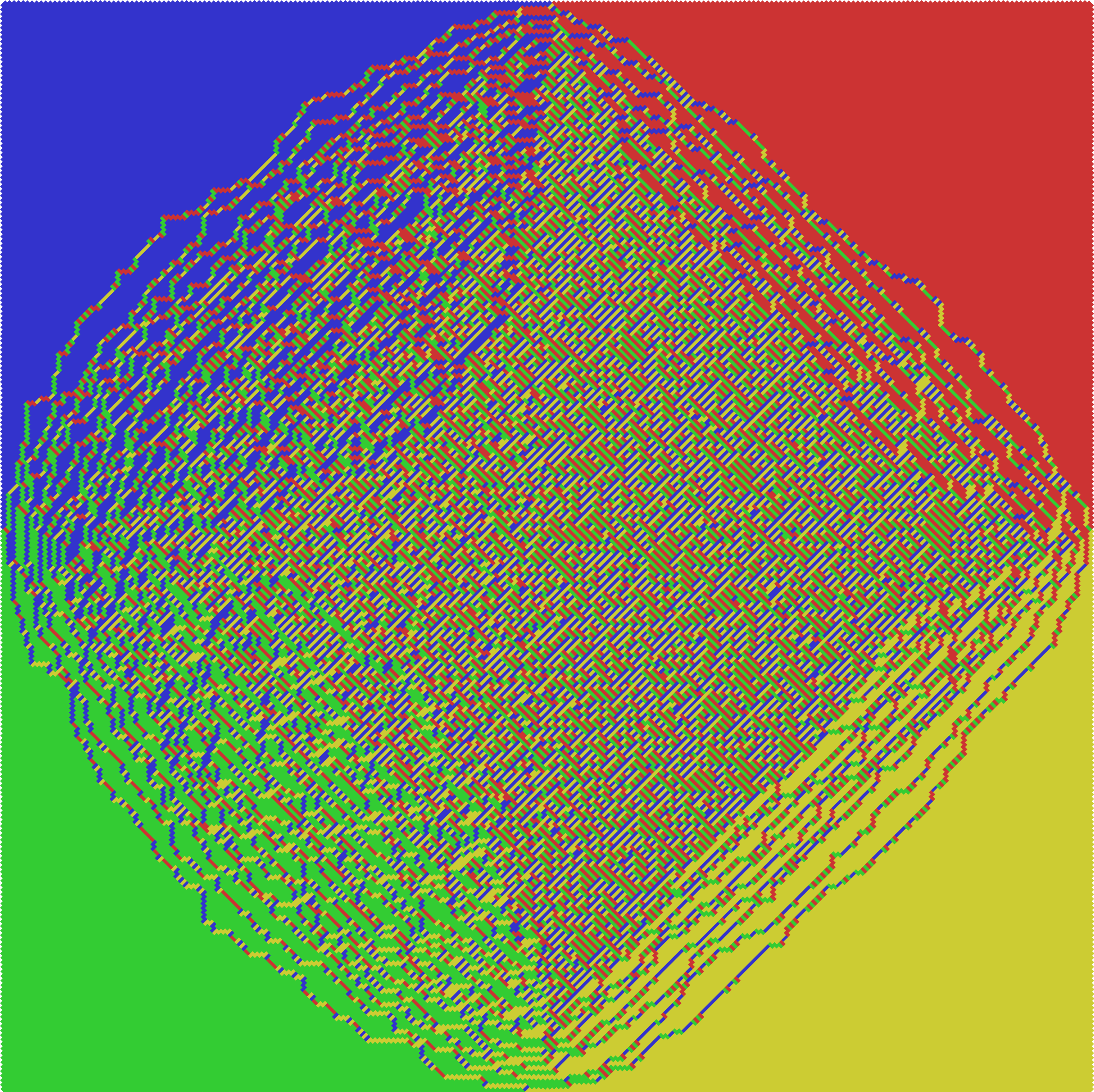}
        \caption{Split two-periodic model with parameters $\alpha = 1/2$ and $\beta=1/5$.}
    \end{subfigure}
    \hspace{1cm}
    \begin{subfigure}[t]{0.3\textwidth}
        \includegraphics[width=\textwidth]{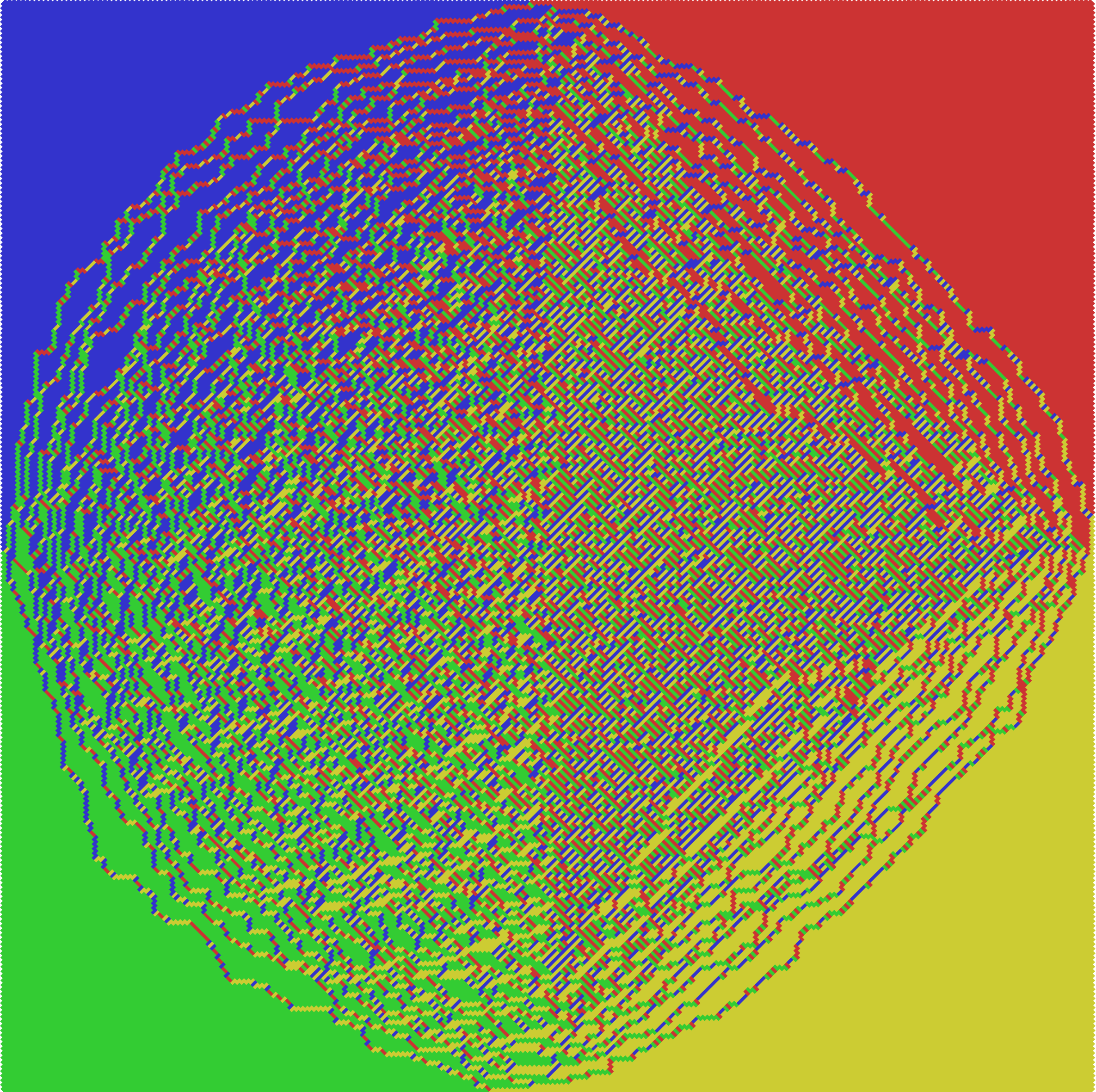}
        \caption{Split two-periodic model with parameters $\alpha = 1$ and $\beta=1/3$.}
    \end{subfigure}
    \caption{Examples of the split two-periodic Aztec diamond model analyzed in \cite{S25}. All simulations throughout this work generated by modifying code provided by Sunil Chhita and Leonid Petrov.}
    \label{fig:split-examples}
\end{figure}

This paper generalizes the split two-periodic model by allowing the interface to occur at any asymptotic coordinate. The main result of this work is Theorem \ref{theorem:kernel}, which provides an integral form of the correlation kernel that is suitable for asymptotic analysis. The proof of this theorem extends the methods of Berggren and Duits \cite{BD19} and is carried out in Section \ref{section:kernel-proof}. Figures \ref{fig:t-split-examples-1} and \ref{fig:t-split-example-2} show example simulations of the $t$-split model under various conditions. 

\begin{figure}[H]
    \centering
    \begin{subfigure}[t]{0.3\textwidth}
        \includegraphics[width=\textwidth]{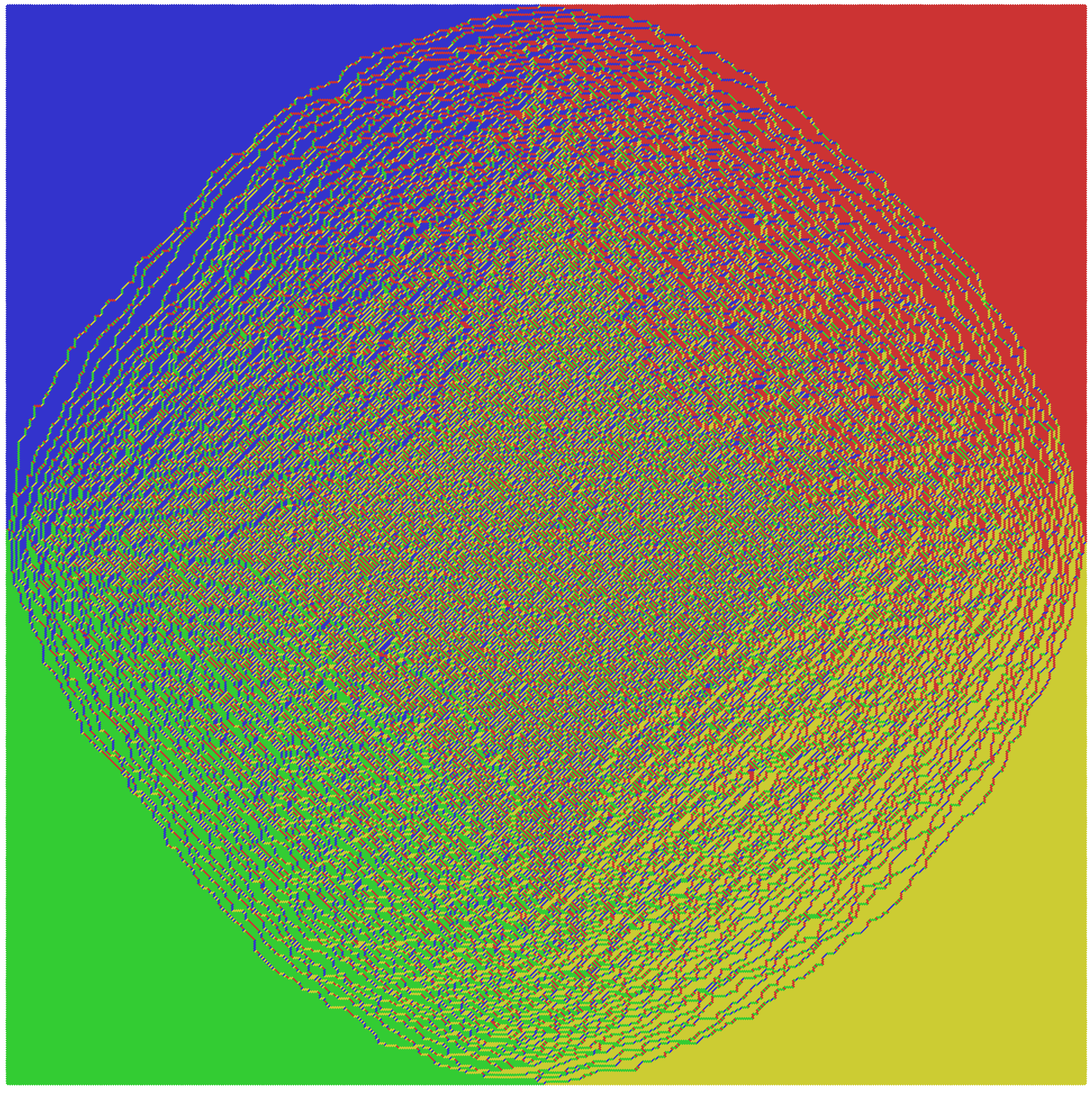}
        \caption{The dominant (right) side of the model has two-periodic parameter $1/2$.}
    \end{subfigure}
    \hspace{1cm}
    \begin{subfigure}[t]{0.3\textwidth}
        \includegraphics[width=\textwidth]{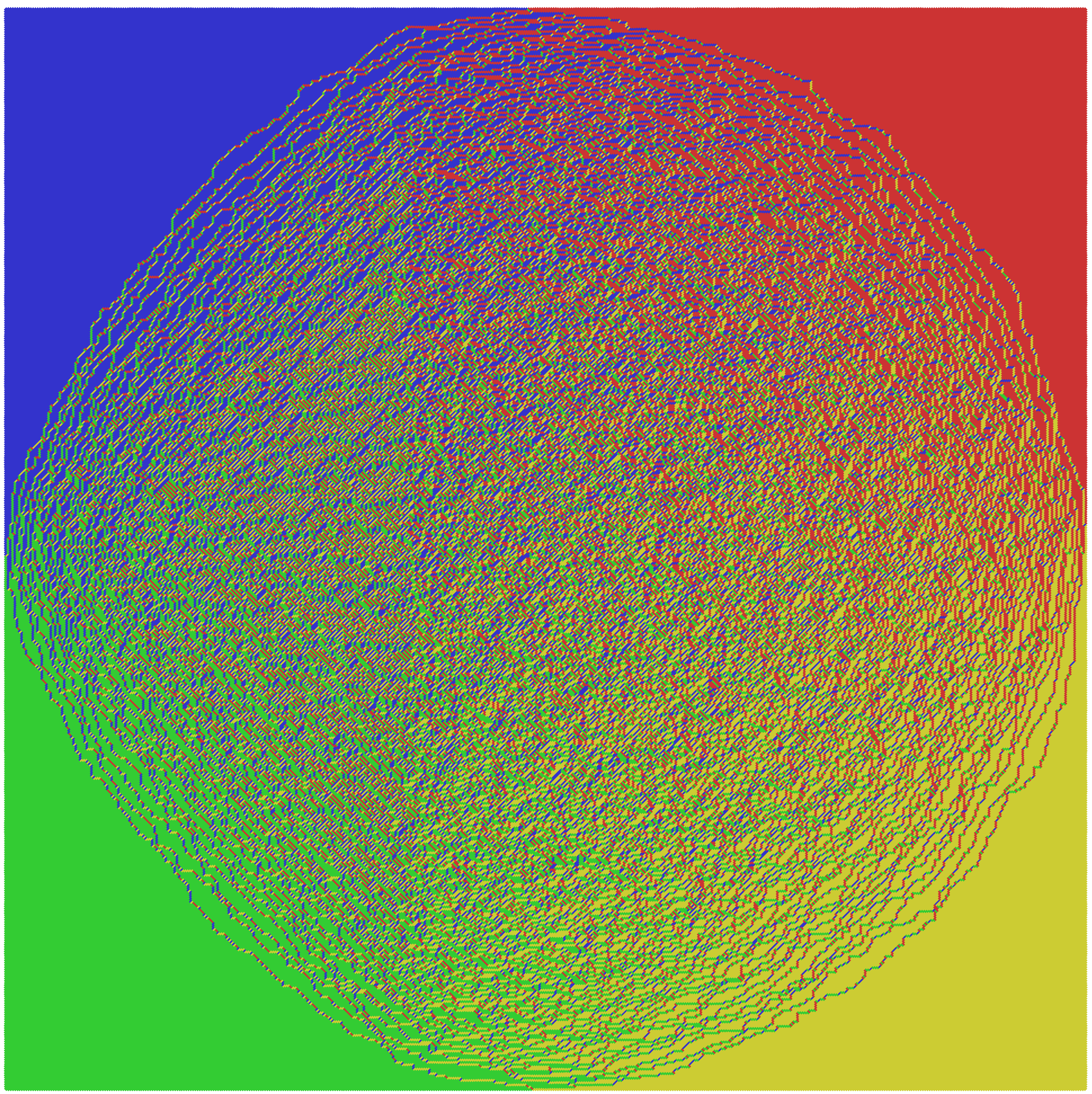}
        \caption{The dominant side has uniform weights.}
    \end{subfigure}
    \caption{Example simulations of the $t$-split two-periodic Aztec diamond model, where $t = 2/5$. In both example the two-periodic parameter on the left is $1/5$, however the limit shapes differ significantly.}
    \label{fig:t-split-examples-1}
\end{figure}

The remainder of the results pertain to the asymptotic analysis of the $t$-split model. When the interface is off-center, the model has a \textit{dominant} side and an \textit{non-dominant} side, each with a corresponding weighting. In this work, the dominant side is always located to the right of the interface. The limit shape of the dominant side of the model is exactly the limit shape of the two-periodic model with the same weighting; this is precisely stated in Proposition \ref{prop:right-side-asymptotics} in Section \ref{section:asymptoticresults}. The non-dominant side of the model exhibits a limit shape that can not be directly related to the two-periodic model. In this paper, we do not prove the limit shape for the non-dominant side of the model in general; however, we conjecture a way to describe it in Conjecture \ref{conj:left-asymptotics}. To describe the limit shape of the non-dominant side of the model, we propose a function that depends on the asymptotic location and whose critical points determine the macroscopic behavior at that location. Similar methods were used in \cite{DK21} to describe the limit shape of the two-periodic model and the necessary background information is given here in Section \ref{section:asymptoticresults}. An example simulation of the $t$-split model, alongside its conjectured limit shape appear in Figure \ref{fig:t-split-example-2}.

\begin{figure}[h]
    \centering
    \begin{subfigure}[t]{0.3\textwidth}
        \includegraphics[width=\textwidth]{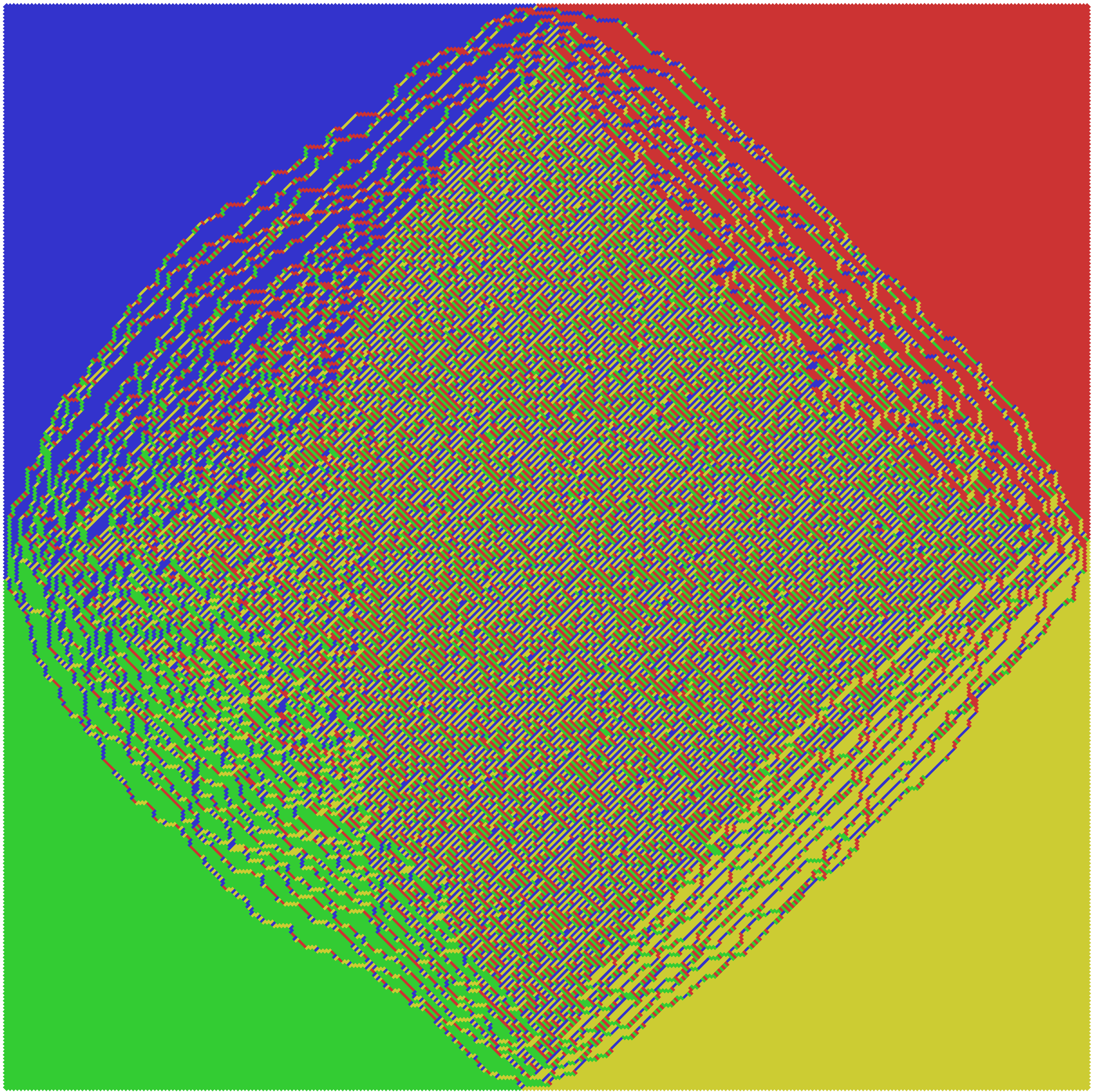}
        \caption{Simulation of the model where $n = 480$.}
        \label{fig:t-split-sim-1}
    \end{subfigure}
    \hspace{1cm}
    \begin{subfigure}[t]{0.3\textwidth}
        \includegraphics[width=\textwidth]{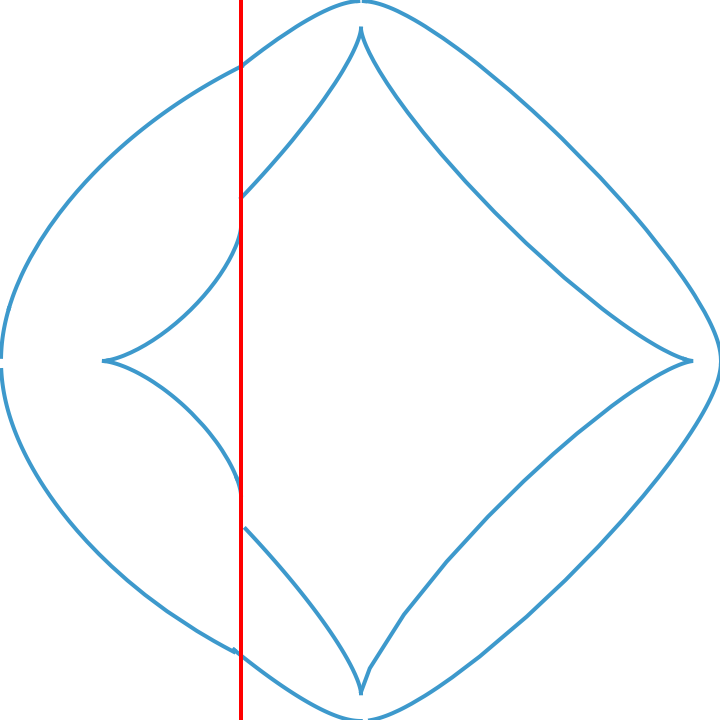}
        \caption{Limit shape produced via Mathematica.}
        \label{fig:t-split-limit}
    \end{subfigure}
    \caption{A simulation of the $t$-split two-periodic Aztec diamond alongside the partially conjectured limit shape. The example has parameters. $\alpha = 1/2$, $\beta = 1/5$ and $t = 1/3$.}
    \label{fig:t-split-example-2}
\end{figure}

Given the conjecture, numerical analysis, and the results from simulations, the behavior on the non-dominant side of the $t$-split model appears to be significantly richer than the split model. The limit shape on the non-dominant side depends on \textit{both} two-periodic parameters as well as the asymptotic location of the interface. Depending on these parameters, there may or may not be a smooth region to the left of the interface. This is illustrated in Figure \ref{fig:t-split-examples-1}.

This work fully proves the limit shape of the model in the case where the dominant side two-periodic parameter $\beta= 0$. This limit was first studied by Johansson and Mason in \cite{JM23} for the two-periodic version of the model.\footnote{We should note that Johansson and Mason analyze the model in a case that translates to $\beta = O(n^{-p})$, where we are only interested in $\beta \to 0$ independent of $n$.} For the two-periodic model, taking this limit (independent of $n$) results in the elimination of the rough region and creation of a frozen-smooth boundary. Since the dominant side of the $t$-split model behaves, in the limit, exactly like its two-periodic counterpart, we see only a frozen and smooth region to the right of the interface. We then prove that the limit shape of the non-dominant side of the model becomes half of a rescaled two-periodic Aztec diamond. An example of $\beta = 0$ is shown in Figure \ref{fig:b=0-example}.
 
\begin{figure}[h]
    \centering
    \begin{subfigure}[t]{0.4\textwidth}
        \includegraphics[width=\textwidth]{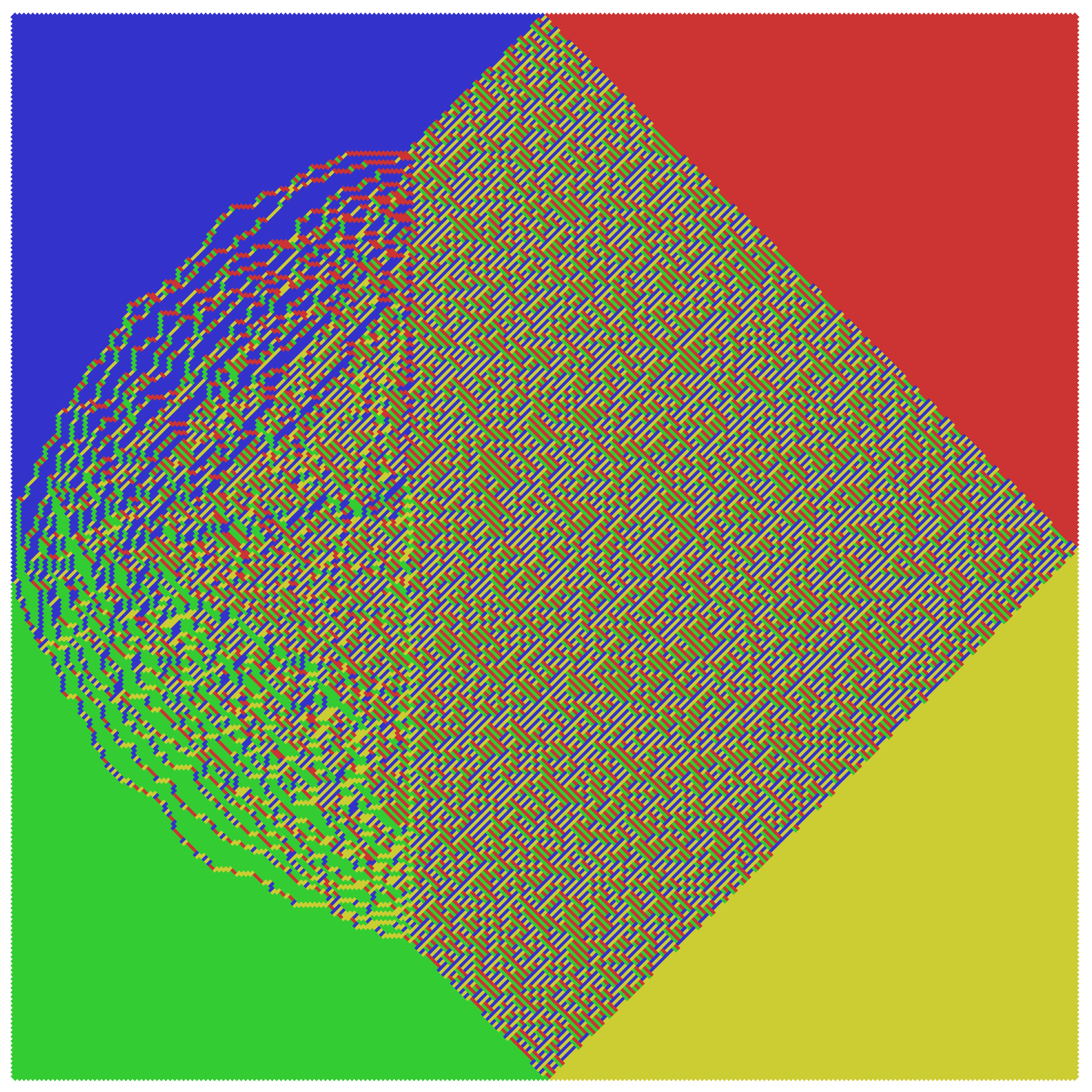}
    \end{subfigure}
    \caption{Simulation of the $t$ split two-periodic Aztec diamond with $\alpha = 1/2$, $\beta = 0$, and $t = 1/3$.}
    \label{fig:b=0-example}
\end{figure}

\subsection{Definition of the Model} \label{section:t-split-def}

We will define the model in the dimer setting, so to start we will detail the coordinate convention for the vertices and edges of the Aztec diamond graph of size $n$. The vertices have a natural bipartite structure $\tV_n^{\text{Az}} = \tW^{\text{Az}}_n \sqcup \tB^{\text{Az}}_n$ where
\begin{equation} \label{eq:coordstart}
    \tW^{\text{Az}}_n = \Big\{ (2j+1,2k) : 0 \leq j \leq n-1 \text{ and } 0 \leq k \leq n \Big\}
\end{equation}
and
\begin{equation} \label{eq:coordend}
    \tB^{\text{Az}}_n = \Big\{ (2j,2k+1) : 0 \leq j \leq n \text{ and } 0 \leq k \leq n-1 \Big\}.
\end{equation}
The edges are given by
\begin{multline} \label{eq:edgecoord}
    \tE^{\text{Az}}_n = \Big\{ ((2j+1, 2k), (2j+1\pm 1, 2k+1)) : 0 \leq j \leq n-1, 0 \leq k \leq n-1 \Big\} \\
    \cup \Big\{ ((2j+1, 2k), (2j+1\pm 1, 2k-1)) : 0 \leq j \leq n-1, 1\leq k \leq n \Big\}.
\end{multline}
Note that this coordinate convention expresses the rotated convention of the Aztec diamond depicted in Figure \ref{fig:AztecDiamondExamples}. Since our weights are two-periodic, we will assume $n = 2N$. The $t$-split two-periodic Aztec diamond has an interface at the vertical line $x = 4tN$. We will restrict to cases where $tN$ is an integer, so that interface falls between the weight repetitions. Let $b = (b_x,b_y) \in \tB^{\text{Az}}_n$ and $w = (w_x,w_y)\in \tW^{\text{Az}}_n$. We say that $b$ and $w$ are neighbors, and write $b \sim w$, if there is an edge between the two vertices. We will use the notation $\text{wt}(b,w)$ to denote the weight of the edge connecting the vertices $b$ and $w$. Note that $\text{wt}(b,w) = \text{wt}(w,b)$. The weight function for the $t$-split two-periodic Aztec diamond is 
\begin{equation} \label{eq:edgewts}
    \text{wt}(b,w) = \begin{cases}
        \alpha^2 & \text{if } w_x \equiv 1 \bmod{4}, \, w_y \equiv 0 \bmod{4}, \, w-b = (\pm1,-1), \text{ and } w_x < 4tN,\\
        \alpha^{-2} & \text{if }  w_x \equiv 1 \bmod{4}, \, w_y \equiv 2 \bmod{4}, \, w-b = (\pm1,-1), \text{ and } w_x < 4tN,\\
        \beta^2 & \text{if } w_x \equiv 1 \bmod{4}, \, w_y \equiv 0 \bmod{4}, \, w-b = (\pm1,-1), \text{ and } w_x > 4tN,\\
        \beta^{-2} & \text{if } w_x \equiv 1 \bmod{4}, \, w_y \equiv 2 \bmod{4}, \, w-b = (\pm1,-1), \text{ and } w_x > 4tN,\\
        1 & \text{if } w_x \equiv 1 \bmod{4}, \text{ and } w-b = (\pm1,1),\\
        1 & \text{if } w_x \equiv 3 \bmod{4}, \text{ and } w-b = (\pm1,\pm1),\\
        0 & \text{otherwise}.
    \end{cases}
\end{equation}
where $0<\alpha,\beta<1$ are constants that we refer to as the \textit{two-periodic parameters}. If $\alpha = \beta$, this model recovers the usual two-periodic Aztec diamond. If $\alpha = \beta = 1$ we recover the uniform Aztec diamond; however we will generally assume that the constants are not equal to 1. The case of $t = 1/2$ was studied in \cite{S25}. In the present work we will primarily focus on the case $t < 1/2$.

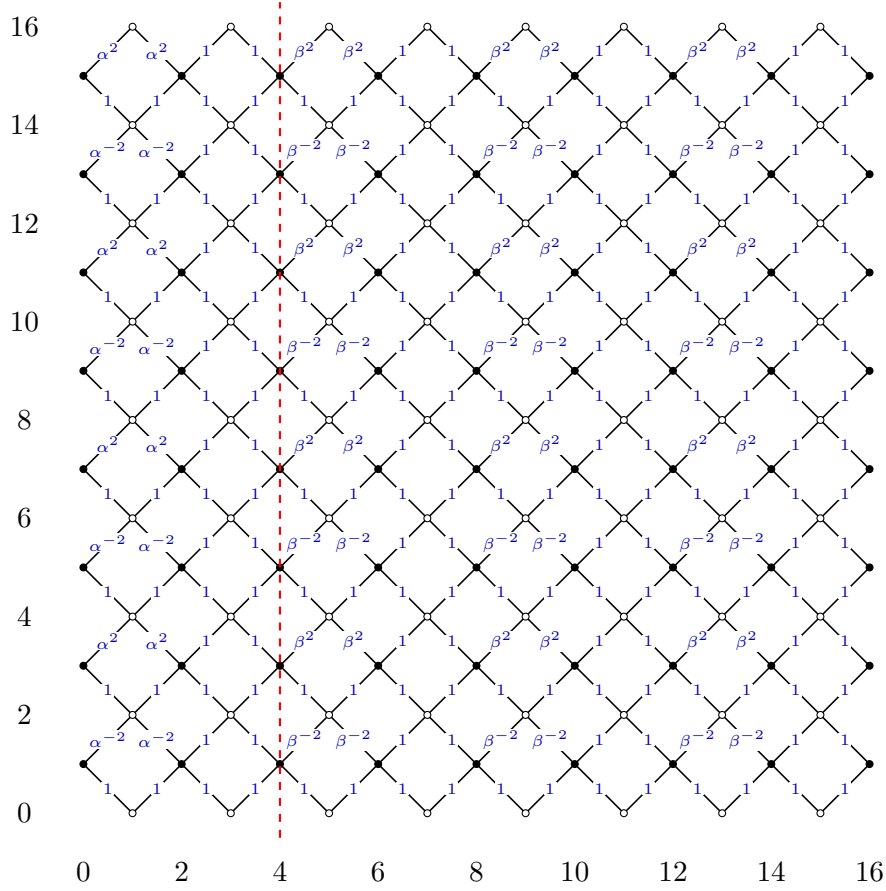
\begin{figure}[h]
    \centering
    \begin{tikzpicture}[scale=0.65]
    \foreach \t in {1,3,5,7,9,11,13,15}
        {
        \draw[semithick] (0,\t) -- (\t,0);
        \draw[semithick] (\t,16) -- (16,\t);
        \draw[semithick] (\t,0) -- (16,16-\t);
        \draw[semithick] (0,\t) -- (16-\t,16);
        }

    \foreach \j in {0,1,2,3,4,5,6,7,8}
        \foreach \k in {0,1,2,3,4,5,6,7}
            {
            \filldraw[black] (2*\j,2*\k+1) circle (2pt);
            }

    \foreach \j in {0,1,2,3,4,5,6,7}
        \foreach \k in {0,1,2,3,4,5,6,7,8}
            {
            \filldraw[color=black,fill=white] (2*\j+1,2*\k) circle (2pt);
            }

    \foreach \j in {0,2,4,6,8,10,12,14,16}
        {
        \node at (\j,-1.2) {$\j$};
        \node at (-1.2,\j) {$\j$};
        }

    \draw[dashed, red, thick] (4,-0.5) -- (4,16.5);

    \foreach \k in {2,4,6,8}
        {
        \node[font=\tiny,blue!70!black,fill=white,inner sep=0.5pt] at (0.5,2*\k-0.5) {$\alpha^{2}$};
        \node[font=\tiny,blue!70!black, fill=white, inner sep=0.5pt] at (1.5,2*\k-0.5) {$\alpha^{2}$};
        }
    \foreach \k in {1,3,5,7}
        {
        \node[font=\tiny,blue!70!black, fill=white, inner sep=0.5pt] at (0.5,2*\k-0.5) {$\alpha^{-2}$};
        \node[font=\tiny,blue!70!black, fill=white, inner sep=0.5pt] at (1.5,2*\k-0.5) {$\alpha^{-2}$};
        }
    \foreach \k in {0,1,2,3,4,5,6,7}
        {
        \node[font=\tiny,blue!70!black, fill=white, inner sep=0.5pt] at (0.5,2*\k+0.5) {$1$};
        \node[font=\tiny,blue!70!black, fill=white, inner sep=0.5pt] at (1.5,2*\k+0.5) {$1$};
        }

    \foreach \px in {5,9,13}
        {
        \foreach \k in {2,4,6,8}
            {
            \node[font=\tiny,blue!70!black, fill=white, inner sep=0.5pt] at (\px-0.5,2*\k-0.5) {$\beta^{2}$};
            \node[font=\tiny,blue!70!black, fill=white, inner sep=0.5pt] at (\px+0.5,2*\k-0.5) {$\beta^{2}$};
            }
        \foreach \k in {1,3,5,7}
            {
            \node[font=\tiny,blue!70!black, fill=white, inner sep=0.5pt] at (\px-0.5,2*\k-0.5) {$\beta^{-2}$};
            \node[font=\tiny,blue!70!black, fill=white, inner sep=0.5pt] at (\px+0.5,2*\k-0.5) {$\beta^{-2}$};
            }
        \foreach \k in {0,1,2,3,4,5,6,7}
            {
            \node[font=\tiny,blue!70!black, fill=white, inner sep=0.5pt] at (\px-0.5,2*\k+0.5) {$1$};
            \node[font=\tiny,blue!70!black, fill=white, inner sep=0.5pt] at (\px+0.5,2*\k+0.5) {$1$};
            }
        }

    \foreach \px in {3,7,11,15}
        {
        \foreach \k in {0,1,2,3,4,5,6,7}
            {
            \node[font=\tiny,blue!70!black, fill=white, inner sep=0.5pt] at (\px-0.5,2*\k+0.5) {$1$};
            \node[font=\tiny,blue!70!black, fill=white, inner sep=0.5pt] at (\px+0.5,2*\k+0.5) {$1$};
            }
        \foreach \k in {1,2,3,4,5,6,7,8}
            {
            \node[font=\tiny,blue!70!black, fill=white, inner sep=0.5pt] at (\px-0.5,2*\k-0.5) {$1$};
            \node[font=\tiny,blue!70!black, fill=white, inner sep=0.5pt] at (\px+0.5,2*\k-0.5) {$1$};
            }
        }

    \end{tikzpicture}
    \caption{The $t$-split two-periodic Aztec diamond of size $n = 2N = 8$ with $t = 1/4$.
    Coordinates are as in \eqref{eq:coordstart}--\eqref{eq:coordend} and weights
    as in \eqref{eq:edgewts}.}
    \label{fig:AztecDiamondExample}
\end{figure}

In \cite{BD19}, Berggren and Duits show that the Aztec diamond is equivalent to a certain non-intersecting paths process. Since the present work builds on their results, the correlation kernel is presented using the coordinate convention from the non-intersecting paths process, which differs from the coordinates presented above. To understand the specific relationship between these models, the reader should refer to \cite{BD19,S25,CD22}. In this work, we simply define a coordinate transformation on the Aztec diamond that will allow us to immediately relate the statement of the kernel to the diamond. 

To transform the coordinates, we start by moving the white vertices of the Aztec diamond to be level with the row of black vertices immediately below them. Consider $w \in \tW^\text{Az}_n$ with initial coordinate $(2j+1,2k)$; it now has the coordinate $(2j+1,2k-1)$. Next, to any vertex $v = (v_x,v_y)$, we make the following change to the $y$-coordinate, 
\[
v_y \mapsto \frac{1}{2}v_y - 2N -\frac{1}{2}.
\]
This transformation is depicted in Figure \ref{fig:AD-alt}. For the Aztec diamond of size $n$ the vertex sets is now expressed at the following
\begin{equation} 
    \tW^{\text{Az-Alt}}_{n} = \Big\{ (2j+1, k) : 0 \leq j < n \text{ and } -n-1 \leq k \leq -1 \Big\},
\end{equation}
and
\begin{equation}
    \tB^{\text{Az-Alt}}_{n} = \Big\{ (2j, k) : 0 \leq j < n \text{ and } -n \leq k \leq -1 \Big\}
\end{equation}
The edges are now given by
\begin{multline} 
    \tE^{\text{Az-Alt}}_{n} = \Big\{\left((2j+1, k),(2j+1 \pm 1, k)\right) : 0 \leq j < n \text{ and } -n \leq k \leq -1\Big\} \\
    \cup \Big\{\left((2j+1, k),(2j + 1 \pm 1, k+1)\right) : 0 \leq j < n \text{ and } -n-1 \leq k \leq -2 \Big\}
\end{multline}
Throughout Section \ref{section:kernelstatement} and the remainder of this paper, we use these transformed coordinates.

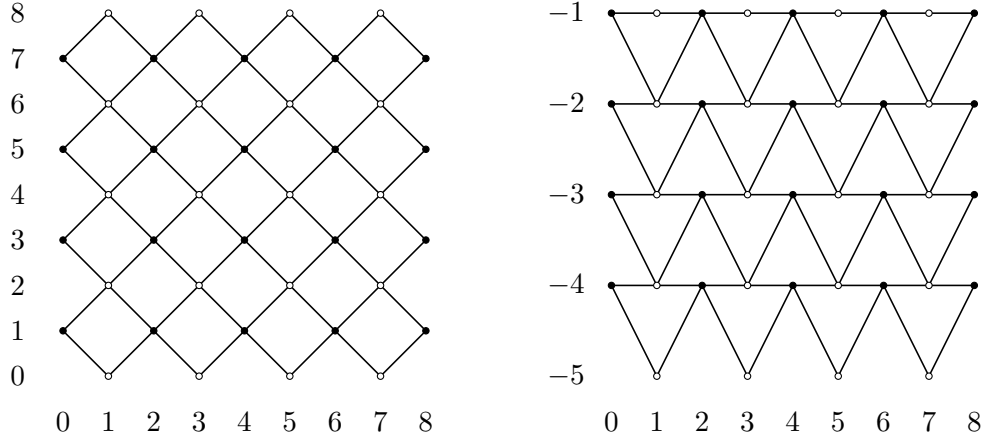
\begin{figure}[h]
    \centering
    \begin{tikzpicture}[scale=0.6]
    \foreach \t in {1,3,5,7}
        {
        \draw[semithick] (0,\t) -- (\t,0);
        \draw[semithick] (\t,8) -- (8,\t);
        \draw[semithick] (\t,0) -- (8,8-\t);
        \draw[semithick] (0,\t) -- (8-\t,8);
        }
    \foreach \j in {0,1,2,3,4}
        \foreach \k in {0,1,2,3}
            {
            \filldraw[black] (2*\j,2*\k+1) circle (2pt);
            }
    \foreach \j in {0,1,2,3}
        \foreach \k in {0,1,2,3,4}
            {
            \filldraw[color=black,fill=white] (2*\j+1,2*\k) circle (2pt);
            }
    \foreach \j in {0,1,2,3,4,5,6,7,8}
        {
        \node at (\j,-1) {$\j$};
        \node at (-1,\j) {$\j$};
        }
    \end{tikzpicture}
    \hspace{1cm}
    \begin{tikzpicture}[scale=0.6]
    \foreach \t in {1,3,5,7}
        {
        \draw[semithick] (0,\t) -- (8,\t);
        \foreach \k in {0,2,4,6}
            {
            \draw[semithick] (\k,\t) -- (\k+1,\t-2);
            }
        \foreach \k in {2,4,6,8}
            {
            \draw[semithick] (\k,\t) -- (\k-1,\t-2);
            }
        }
    \foreach \j in {0,1,2,3,4}
        \foreach \k in {0,1,2,3}
            {
            \filldraw[black] (2*\j,2*\k+1) circle (2pt);
            }
    \foreach \j in {0,1,2,3}
        \foreach \k in {0,1,2,3,4}
            {
            \filldraw[color=black,fill=white] (2*\j+1,2*\k-1) circle (2pt);
            }
    \foreach \j in {0,1,2,3,4,5,6,7,8}
        {
        \node at (\j,-2) {$\j$};
        }
    \foreach \j in {-1,-2,-3,-4,-5}
        {
        \node at (-1,2*\j+9) {$\j$};
        }
    \end{tikzpicture}
    \caption{The above depicts the coordinate changes on the Aztec diamond of size $4$.}
    \label{fig:AD-alt}
\end{figure}

\section{Results} 

\subsection{$t$-Split Correlation Kernel} \label{section:kernelstatement}
The goal of this section is to state the correlation kernel for the $t$-split two-periodic Aztec diamond under the assumption that $t < 1/2$. 
\begin{remark}
    Because of the interface, there is not a symmetry argument between the two discrete models, $t < 1/2$ and $t > 1/2$, when the weighting scheme are swapped. In \cite{S25}, there is a description between the asymptotic symmetry of the models. We conjecture the same is true for this generalize model, however at this time it has not been carefully checked. 
\end{remark}
Before we are able to state the correlation kernel, we must first define some necessary preliminary functions. Most of the notation here is unchanged from \cite{S25}. Let
\begin{equation}
    \Phi_\eps(z) = \frac{1}{(1-z^{-1})^2}
    \begin{pmatrix}1 & \eps^2 z^{-1} \\ \eps^{-2} & 1\end{pmatrix}^2
    \begin{pmatrix}1 & z^{-1} \\ 1 & 1 \end{pmatrix}^2.
\end{equation}
The origins of this matrix and how it relates to the model can be found in \cite[Theorem 5.2]{BD19}. Additionally, for the $t$-split two periodic Aztec diamond we define 
\begin{equation} \label{eq:phi-m'-to-m}
    \phi_{m'\to m, N}(z) = 
    \begin{cases}
        \Phi_\alpha(z)^{m-m'} & m' < m < tN, \\
        \Phi_\alpha(z)^{tN-m'}\Phi_\beta(z)^{m-tN} & m' < tN < m, \\
        \Phi_\beta(z)^{m-m'} & tN < m' < m.
    \end{cases}
\end{equation}
The eigenvalues of $\Phi_\eps(z)$ are 
\begin{equation} \label{eq:r1}
    r_{\eps,1}(z) = \frac{1}{(z-1)^2}\left((z+1)^2+2z(\eps^2 + \eps^{-2}) + 2 (\eps + \eps^{-1}) \sqrt{z^3+(\eps^2+\eps^{-2})z^2+z} \right)
\end{equation} 
and
\begin{equation} \label{eq:r2}
    r_{\eps,2}(z) = \frac{1}{(z-1)^2}\left((z+1)^2+2z(\eps^2 + \eps^{-2}) - 2 (\eps + \eps^{-1}) \sqrt{z^3+(\eps^2+\eps^{-2})z^2+z} \right).
\end{equation} 
We chose the branch cuts for the above to be $(-\infty,-\eps^{-2}] \cup [-\eps^2,0]$ and for $z > 0$ we choose the positive square root. We also adopt the notation
\begin{equation} \label{eq:f1}
    F_{\eps,1}(z) = E_\eps(z) \begin{pmatrix}1 & 0 \\ 0 & 0\end{pmatrix} E_\eps(z)^{-1} = \begin{pmatrix}
        \frac{1}{2} - \frac{z(\eps^2-1)}{2\sqrt{z(z+\eps^2)(1+\eps^2 z)}} & -\frac{\eps^2(z+1)}{2\sqrt{z(z+\eps^2)(1+\eps^2z)}} \\
        -\frac{z(z+1)}{2\sqrt{z(z+\eps^2)(1+\eps^2z)}} & \frac{1}{2} + \frac{z(\eps^2-1)}{2\sqrt{z(z+\eps^2)(1+\eps^2 z)}}
    \end{pmatrix} 
\end{equation} 
and 
\begin{equation} \label{eq:f2}
    F_{\eps,2}(z) = E_\eps(z) \begin{pmatrix}0 & 0 \\ 0 & 1\end{pmatrix} E_\eps(z)^{-1} = \begin{pmatrix}
        \frac{1}{2} + \frac{z(\eps^2-1)}{2\sqrt{z(z+\eps^2)(1+\eps^2 z)}} & \frac{\eps^2(z+1)}{2\sqrt{z(z+\eps^2)(1+\eps^2z)}} \\
        \frac{z(z+1)}{2\sqrt{z(z+\eps^2)(1+\eps^2z)}} & \frac{1}{2} - \frac{z(\eps^2-1)}{2\sqrt{z(z+\eps^2)(1+\eps^2 z)}}
    \end{pmatrix},
\end{equation}
which uses the same branch cuts as above. Lastly we define the function,
\begin{equation} \label{eq:g}
    g_{\alpha,\beta}(z) = \frac{2z(1+\alpha^2\beta^2)+\beta^2(z^2+1)+\alpha^2(z^2+1)}{4\sqrt{(z+\alpha^2)(1+\alpha^2z)(z+\beta^2)(1+\beta^2z)}}.
\end{equation}
Note that $g_{\alpha,\beta}(z) = g_{\beta,\alpha}(z)$, and $g_{\alpha,\alpha}(z) = 1/2$. Assuming $\beta > \alpha$, we define $g_{\alpha,\beta}(z)$ to have the branch cuts $[-\alpha^{-2},-\beta^{-2}] \cup [-\beta^2,-\alpha^2]$. Additionally we define
\begin{equation} \label{eq:c-def}
    c(z) = \frac{2}{1+2g_{\alpha,\beta}(z)}
\end{equation}
for brevity. We are now prepared to state the correlation kernel. Note that the correlation kernel is presented in the form of a $2\times 2$ matrix. 
\begin{theorem} \label{theorem:kernel}
    Let $-N \leq \xi,\xi' \leq -1$ and $0 < m, m' < N$. The correlation kernel $t$-split two-periodic Aztec diamond of size $2N$ with $t < 1/2$ can be written as a piecewise expression depending on whether $m' \leq tN$ or $m' > tN$. For $0 < m' \leq tN$ the kernel is:
    \begin{multline} \label{eq:kernel-left}
        \left[\K_N(4m',2\xi'+j;4m,2\xi+i)\right]_{i,j=0,1} = -\frac{\1_{m>m'}}{2\pi \ii} \oint_{\gamma_{0,1}}\frac{\dz}{z} z^{\xi'-\xi}\phi_{m'\to m,N}(z)\\
        +\frac{1}{(2\pi\ii)^2}\oint_{\gamma_1} \dw \oint_{\gamma_{0,1}} \frac{\dz}{z(z-w)}\frac{w^{\xi'+N}(z-1)^N}{z^{\xi+N}(w-1)^N}r_{\alpha,1}(w)^{tN-m'}F_{\alpha,1}(w) \\
        \times \Phi_{\beta}(w)^{(1/2-t)N} \Phi_{\beta}(z)^{(t-1/2)N}\Phi_{\alpha}(z)^{-tN}\phi_{0\to m,N}(z) \\
        + \frac{1}{(2\pi\ii)^2}\oint_{\gamma_1}\dw\oint_{\gamma_{0,1}} \frac{\dz}{z(z-w)}\frac{w^{\xi'+N}(z-1)^N}{z^{\xi+N}(w-1)^N}r_{\alpha,2}(w)^{tN-m'} c(w) F_{\alpha,2}(w) F_{\beta,1}(w)F_{\alpha,1}(w) \\
        \Phi_{\beta}(w)^{(1/2-t)N} \Phi_{\beta}(z)^{(t-1/2)N} \Phi_{\alpha}(z)^{-tN}\phi_{0\to m,N}(z).
    \end{multline}
    For $tN < m' < N$ the kernel is:
    \begin{multline} \label{eq:kernel-right}
        \left[\K_N(4m',2\xi'+j;4m,2\xi+i)\right]_{i,j=0,1} = - \frac{\1_{m>m'}}{2\pi\ii} \oint_{\gamma_{0,1}}\frac{\dz}{z} z^{\xi'-\xi} \phi_{m' \to m,N}(z) \\
        + \frac{1}{(2\pi\ii)^2}\oint_{\gamma_1} \dw \oint_{\gamma_{0,1}} \frac{\dz}{z(z-w)}\frac{w^{\xi'+N}(z-1)^N}{z^{\xi+N}(w-1)^N}r_{\beta,1}(w)^{N/2-m'} F_{\beta,1}(w)\\
        \Phi_{\beta}(z)^{(t-1/2)N}\Phi_{\alpha}(z)^{-tN}\phi_{0\to m,N}(z) \\
        + \frac{1}{(2\pi\ii)^2}\oint_{\gamma_1}\dw\oint_{\gamma_{0,1}} \frac{\dz}{z(z-w)}\frac{w^{\xi'+N}(z-1)^N}{z^{\xi+N}(w-1)^N}r_{\beta,1}(w)^{(2t - 1/2)N - m'}c(w)F_{\beta,1}(w)F_{\alpha,1}(w)F_{\beta,2}(w)\\
        \times \Phi_{\beta}(z)^{(t-1/2)N}\Phi_{\alpha}(z)^{-tN}\phi_{0\to m,N}(z). 
    \end{multline}
    In both cases, $\gamma_1$ is a simple, positively oriented contour surrounding $1$ but not $0$, while $\gamma_{0,1}$ is a simple, positively oriented contour surrounding both $0$ and $1$ with $\gamma_1$ in its interior.
\end{theorem}
The above kernel recovers the split two-periodic kernel when $t = 1/2$ and it recovers the two-periodic kernel when $\alpha=\beta$. The proof of this theorem is presented in Section \ref{section:kernel-proof} with some technical details deferred to Section \ref{section:technical-lemmas}. 
\begin{remark} \label{rmk:two-per-comparision}
    It is useful to relate the above expression to the correlation kernel of the two-periodic Aztec diamond. Let $\K_{TP,N}^{\eps}$ denote the kernel for the two-periodic Aztec diamond with weight parameter $\eps$. If $m,m' > tN$, the correlation for the $t$-split two-periodic model can be expressed as 
    \begin{multline}
        \left[\K_N(4m',2\xi'+j;4m,2\xi+i)\right]_{i,j=0,1} = \left[\mathbb{K}_{TP,N}^{\beta}(4m',2\xi'+j;4m,2\xi+i)\right]_{i,j=0,1} \\
        + \frac{1}{(2\pi\ii)^2}\oint_{\gamma_1}\dw\oint_{\gamma_{0,1}} \frac{\dz}{z(z-w)}\frac{w^{\xi'+N}(z-1)^N}{z^{\xi+N}(w-1)^N}r_{\beta,1}(w)^{(2t - 1/2)N - m'}c(w) \\
        \times F_{\beta,1}(w)F_{\alpha,1}(w)F_{\beta,2}(w)\Phi_{\beta}(z)^{m-N/2},
    \end{multline}
    so to analyze the asymptotics we only need to focus on the last integral. This is not the case for $m,m' < tN$, so the asymptotics prove to be a more laborsome task. 
\end{remark}

\subsection{Asymptotics} \label{section:asymptoticresults}

Let $(x,y) \in [0,1] \times [-1,0]$ denote the \textit{asymptotic coordinate} of $(4m,2\xi+i)$, where $m  = xN + \mu$, $\xi = yN + \sigma$ and $\mu, \sigma = o(N)$. Given an asymptotic coordinate, it is natural to ask which macroscopic region--smooth, rough, or frozen--the point is located in or whether it lies on some boundary between regions. This analysis constitutes the first level of understanding with respect to the limit shape of the model.

Duits and Kuijlaars neatly relate the macroscopic boundary of the two-periodic Aztec diamond to a saddle function defined on a Riemann surface in \cite{DK21}. Consider the two-periodic Aztec diamond with parameter $\alpha$. The related Riemann surface $\mathcal{R}$ consists of two copies of $\mathbb{C}\backslash \left((-\infty,-\alpha^{-2}] \cup [-\alpha^2,0]\right)$ glued together along the branch cuts $(-\infty,-\alpha^{-2}]$ and $[-\alpha^2,0]$ in the usual crosswise manner. This surface has genus 1 unless $\alpha = 1$, in which case it has genus 0. The corresponding saddle function is then 
\begin{equation}
    \Psi_\alpha(z;x,y) = (1+y)\log z - \log(z-1) + \left(\frac{1}{2}-x\right) \log \mathbf{r_\alpha}(z),
\end{equation}
where $\mathbf{r_\alpha}(z) = r_{\alpha,1}(z)$ on the first sheet of $\mathcal{R}$ and $\mathbf{r_\alpha}(z) = r_{\alpha,2}(z)$ on the second sheet of $\mathcal{R}$. Recall that the functions $r_{\alpha,i}(z)$ are defined back in \eqref{eq:r1} and \eqref{eq:r2}. The differential of $\Psi_\alpha$ is single-valued and meromorphic,
\begin{equation}
    \Psi'_\alpha(z) \, dz = \left(\frac{1+y}{z} - \frac{1}{z} + \left(\frac{1}{2} - x\right) \frac{\mathbf{r_\alpha}'(z)}{\mathbf{r_\alpha}(z)} \right) \, dz.
\end{equation}
Counting multiplicities, the above differential has four zeros. The zeros of this differential are referred to as the saddle points of $\Psi_\alpha$. 

Let $\mathcal{C}_1$ denote the union of the intervals $[-\alpha^{-2},-\alpha^2]$ on both sheets of $\mathcal{R}$, and let $\mathcal{C}_2$ denote the union of the intervals $[0,\infty)$ on both sheets of $\mathcal{R}$.\footnote{This is the same notation used in \cite{DK21}.} The union $\mathcal{C}_1 \cup \mathcal{C}_2$ is referred to as the \textit{real part} of $\mathcal{R}$. It was shown in \cite{DK21} that for any fixed $(x,y)$, there are always at least two distinct saddle points on $\mathcal{C}_1$. The location of the remaining saddle points completely determines the macroscopic region of the point $(x,y)$.
\begin{definition}[Macroscopic regions of the two-periodic Aztec diamond] \label{def:macro-regions}
    Fix $(x,y) \in [0,1] \times [-1,0]$. 
    \begin{enumerate}[label=(\roman*)]
        \item If there are two simple saddle points on $\mathcal{C}_2$, then $(x,y)$ is in the \textit{frozen region}. 
        \item If there are two simple saddles points not in the real part of $\mathcal{R}$, then $(x,y)$ is in the \textit{rough region}.
        \item If there are two simple saddles on $\mathcal{C}_1$, then $(x,y)$ is in the \textit{smooth region}. 
    \end{enumerate}
    The boundary of regions occurs when saddle points coalesce:
    \begin{enumerate}[label=(\roman*)]
        \setcounter{enumi}{3}
        \item If there is a double saddle point on $\mathcal{C}_2$, then $(x,y)$ is on the \textit{frozen-rough boundary} of the model. 
        \item If there is a double or triple saddle point on $\mathcal{C}_1$, then $(x,y)$ is  on the \textit{rough-smooth boundary} of the model. 
    \end{enumerate}
\end{definition}
A similar definition was given in \cite{S25} for the split two-periodic Aztec diamond. Since the correlation kernel for the split two-periodic model has a piecewise definition, it is natural to describe the macroscopic boundary via a piecewise saddle function:
\begin{equation} \label{eq:saddle-split}
    \psi_{\alpha,\beta,i}(z;x,y) = 
    \begin{cases}
        (1+y)\log z - \log(z-1) + \left(\frac{1}{2}-x\right) \log r_{\alpha,i}(z) & x \leq 1/2,\\
        (1+y)\log z - \log(z-1) + \left(\frac{1}{2}-x\right) \log r_{\beta,i}(z) & x > 1/2,
    \end{cases}
\end{equation}
where $\alpha$ and $\beta$ are the two-periodic parameters to the left and right of the interface, respectively. That is, the saddle function on either side of the interface coincides with the saddle function of the corresponding two-periodic model. Note that the saddle function above is described sheet by sheet, rather than on a single Riemann surface. This is because there is currently no single Riemann surface that is consistent across the interface. For fixed $(x,y)$, one can realize the same Riemann surface as the two-periodic case and use Definition \ref{def:macro-regions} to describe the regions of the model. For fixed $z$, the saddle function in \eqref{eq:saddle-split} is continuous, but not necessarily differentiable, as we vary $(x,y)$ across the interface. In particular, certain points along the interface may lie on the boundary between macroscopic regions.\\

Since the $t$-split correlation kernel is also piecewise, we describe the macroscopic boundary separately on either side of the interface. Below we state what is known and conjectured regarding the asymptotics of the model in each case.
\begin{proposition}[Macroscopic boundary to the right of the interface] \label{prop:right-side-asymptotics}
    The macroscopic boundary of the $t$-split two-periodic Aztec diamond to the right of the interface is governed by the saddle functions
    \begin{equation}
        \psi_{\beta,i}(z;x,y) = (1+y)\log z - \log(z-1) + \left(\frac{1}{2}-x\right) \log r_{\beta,i}(z), 
    \end{equation}
    for $i = 1,2$. In other words, the macroscopic boundary to the right of the interface is identical to that of the two-periodic model with parameter $\beta$.
\end{proposition}
The proof of this result is found in Section \ref{section:asymptotics}. As stated in Remark \ref{rmk:two-per-comparision}, the correlation kernel for two coordinates to the right of the interface can be written as the two-periodic correlation kernel plus an additional double contour integral, 
\begin{multline} \label{eq:right-side-extra}
    \left[I_{\beta,N}(4m',2\xi'+j;4m,2\xi+i)\right]_{i,j=0,1} \\
        = \frac{1}{(2\pi\ii)^2}\oint_{\gamma_1}\dw\oint_{\gamma_{0,1}} \frac{\dz}{z(z-w)}\frac{w^{\xi'+N}(z-1)^N}{z^{\xi+N}(w-1)^N}r_{\beta,1}(w)^{(2t - 1/2)N - m'}c(w) \\
        \times F_{\beta,1}(w)F_{\alpha,1}(w)F_{\beta,2}(w)\Phi_{\beta}(z)^{m-N/2}.
\end{multline}
To prove Proposition \ref{prop:right-side-asymptotics}, it suffices to show that $I_{\beta,N} \to 0$ as $N \to \infty$ for any points strictly to the right of the interface. This proof, in fact, reveals an even strong statement about the contour integral $I_{\beta,N}$:
\begin{corollary}
    For $m,m' > tN$, there exists $c > 0$ such that $I_{\beta,N} = O(e^{-cN})$. 
\end{corollary}
This is notably different than the results obtained for the split two-periodic model in \cite{S25}, where we saw that the second order asymptotic term in the split two-periodic model was, in certain regions of the model, different in order from the two-periodic model. The above corollary shows that the right side of the $t$-split model agree in both the first and second order with the two-periodic model. \\

The story to the left of the interface is more complicated. First, we restrict the kernel presented in \eqref{eq:kernel-left} to the case where $m < tN$. Making some analytic simplifications yields
\begin{lemma} \label{lem:simplified-left-kernel}
    For $m,\,m' < tN$, the correlation kernel of the $t$-split model can be written as
    \begin{multline} \label{eq:kernel-left-simplified}
        \left[\K_N(4m',2\xi'+j;4m,2\xi+i)\right]_{i,j=0,1} = \frac{\1_{m\leq m'}}{2\pi \ii} \oint_{\gamma_{0,1}}\frac{\dz}{z} z^{\xi'-\xi}\Phi_{\alpha}(z)^{m-m'} \\
        -\frac{1}{(2\pi\ii)^2}\oint_{\gamma_1}\dw\oint_{\gamma_{0,1}} \frac{\dz}{z(z-w)}\frac{w^{\xi'+N}(z-1)^N}{z^{\xi+N}(w-1)^N}r_{\alpha,2}(w)^{tN-m'}r_{\beta,2}(w)^{(1/2-t)N} \\
        \times c(w) F_{\alpha,2}(w) F_{\beta,2}(w) \Phi_{\beta}(z)^{(t-1/2)N} \Phi_{\alpha}(z)^{m-tN}.
    \end{multline}
\end{lemma}
The proof of the above lemma is postponed to Section \ref{section:technical-lemmas}. The form of this kernel motivates the following conjecture:
\begin{conj}[Macroscopic boundary to the left of the interface] \label{conj:left-asymptotics}
    The macroscopic boundary of the $t$-split two-periodic Aztec diamond to the left of the interface is governed by the saddle function
    \begin{equation}
        \varphi_{\alpha,\beta}(z;x,y) = (1+y) \log z - \log(z-1) + (t-x) \log r_{\alpha,2}(z) + (1/2 - t) \log r_{\beta,2}(z),
    \end{equation}  
    in the sense that the macroscopic boundary corresponds to the coalescing of critical points, in the same way as in the two-periodic model.
\end{conj}
The use of $r_{\alpha,2}(z)$ and $r_{\beta,2}(z)$ follows from the convention chosen for the square root and logarithm branch cuts; an alternative choice may alter the second indices. Proving the above conjecture involves understanding, to an extent, the location and movement of the critical points of $\varphi_{\alpha,\beta}$. From this information, one can then use the method of steepest descent to approximate the kernel for large $N$. The critical points of $\varphi_{\alpha,\beta}(z;x,y)$ can not be computed explicitly for general $(x,y)$ and is the main source of tediousness in proving the stated claim. For explicit parameters we are able to understand the critical points of $\varphi_{\alpha,\beta}$ numerically. Numerical results allow us to produce sample limit shapes such as the one presented in Figure \ref{fig:t-split-example-2}.

Based on the form of the function in Conjecture \ref{conj:left-asymptotics}, it is clear that the limit shape to the left of the interface now depends on the parameters $\alpha$, $\beta$, and $t$. For $0 < \alpha < 1$, numerical evidence suggest that the parameters $\beta$ and $t$ determine whether or not there is a smooth region to the left of the interface. While the location of the cusp depends on $\alpha$, we conjecture that the existence of the cusp only depends on $\beta$ and $t$. We state this precisely in the following conjecture, 
\begin{conj} \label{conj:left-smooth}
    Let $0 < \beta, \,\alpha < 1$. Then a smooth region exists to the left of the interface if and only if
    \[
    \frac{\beta^2}{1+\beta^2} < t.
    \]
\end{conj}
Given a two-periodic model with parameter $\beta = p/q$ such that $0<\beta<1$, the left most \textit{cusp} of the rough-smooth boundary occurs at the horizontal position $p^2/(p^2+q^2)\in (0,1/2)$. Therefore, the statement in the conjecture is simply saying that the cusp for the two-periodic model with parameter $\beta$ would occur before the interface line. Note that we do not have a similar condition on left side parameter $\alpha$. Figure \ref{fig:ex-4} gives an example where the left side cusp would happen after the interface in the two-periodic model; however, since $\beta$ is small, and thus the model presents with a large smooth region, we still see a smooth region to the left of the interface. 

To prove Conjecture \ref{conj:left-smooth}, we once again need a certain level understanding on how the location of the critical points of $\varphi_{\alpha,\beta}$ depend on $\alpha$, $\beta$, and $t$. Since the critical points cannot be computed exactly, we require an implicit proof of the fact. In this work, we do wish to present numerical evidence. Consider the the function $\varphi_{\alpha,\beta}$ for $\alpha = 1/2$ and $t = 1/4$. For any $\beta$, $x$, and $y$, we can compute the critical points of $\varphi_{\alpha,\beta}$ numerically using the Mathematica function \texttt{NSolve[$\cdot$]} and then use Definition \ref{def:macro-regions} to determine the macroscopic region of the model. This is done in Table \ref{tab:num-conj-2} for $x = 1/4$ and various values of $y$ and $\beta$. The $x$ parameter was chosen to be exactly $1/4$ to ensure we detect small smooth regions. The numerical analysis shows that the smooth region to the left of the interface disappears between $\beta = 0.5$ and $\beta = 0.6$. This aligns with the condition outlined in Conjecture \ref{conj:left-smooth}. Note that the numerical results do not suggest the elimination of the frozen region when $\beta \geq 0.5$, just that it lies above and below the lines $y = -0.1$ and $y = -0.9$, respectively. \\

\begin{figure}[h]
    \centering
    \begin{subfigure}[t]{0.3\textwidth}
        \includegraphics[width=\textwidth]{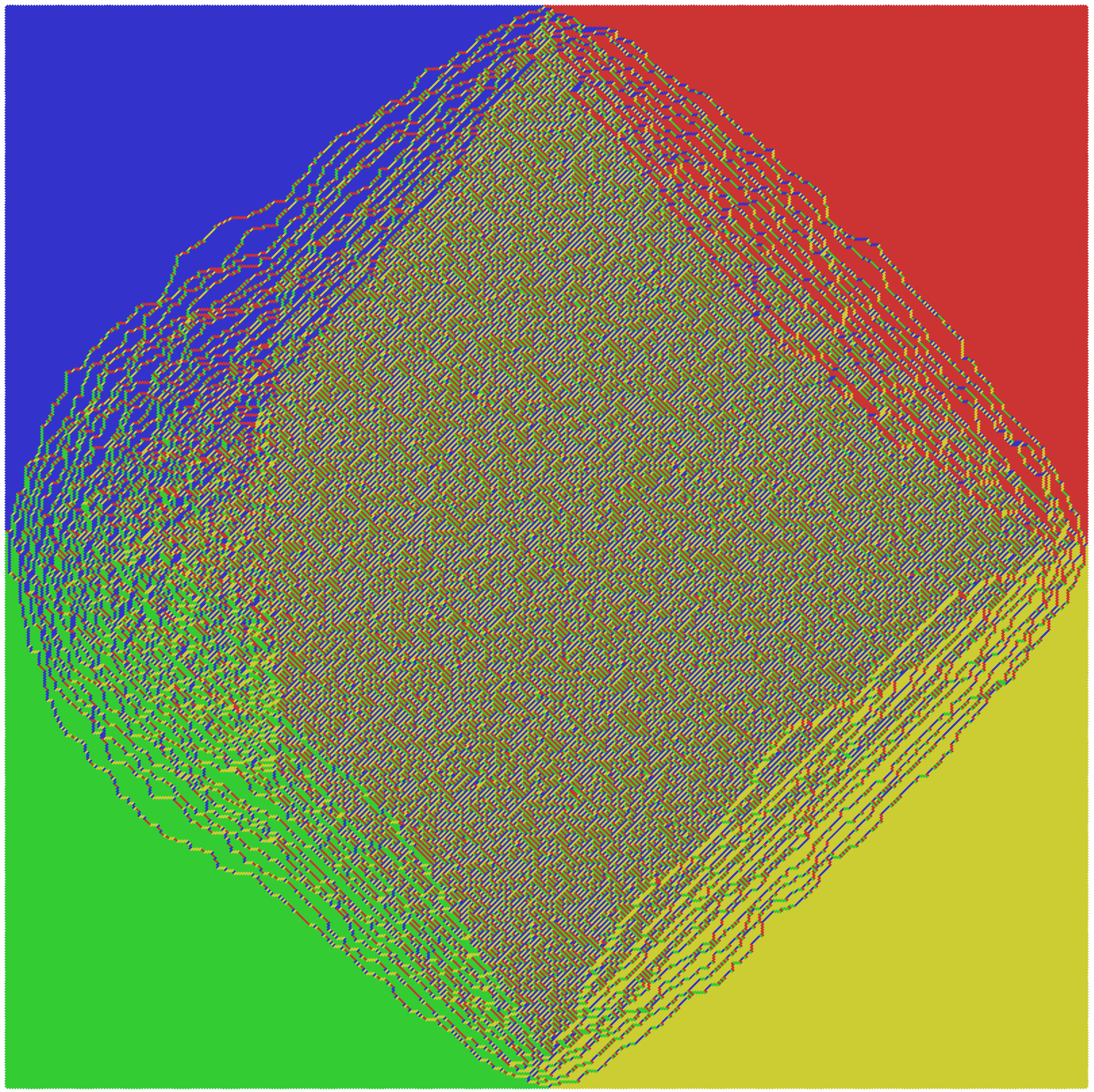}
        \caption{Simulated model with $t = 1/4$, $\alpha = 2/3$, and $\beta = 1/5$.}
        \label{fig:ex-4}
    \end{subfigure}
    \hspace{1cm}
    \begin{subfigure}[t]{0.3\textwidth}
        \includegraphics[width=\textwidth]{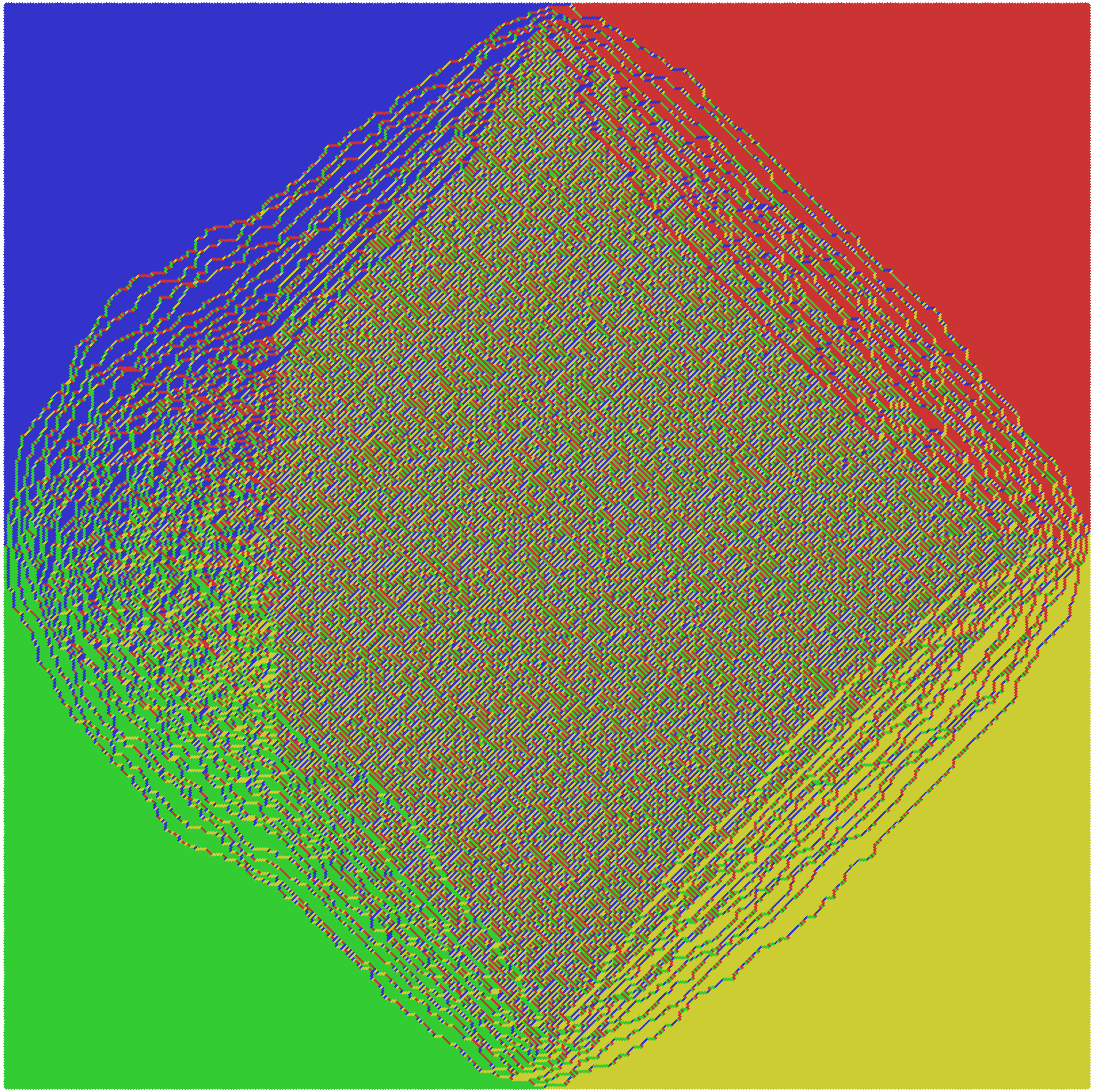}
        \caption{Simulated model with $t = 1/4$, $\alpha = 1$, and $\beta = 1/5$.}
        \label{fig:ex-5}
    \end{subfigure}
    \caption{Simulated models to help illustrate Conjecture \ref{conj:left-smooth}. In (a), since $\beta$ is sufficiently small and $0 < \alpha < 1$, the model still exhibits a smooth region to the left of the interface. This is no longer the case in (b) because $\alpha = 1$.}
    \label{fig:ex-4-5}
\end{figure}

\begin{table}[ht]
  \centering
  \begin{tabular}{c cccccc}
    \toprule
    & \multicolumn{6}{c}{$\beta$} \\
    \cmidrule(lr){2-7}
    $y$ & 0.1 & 0.2 & 0.3 & 0.4 & 0.5 & $\geq$ 0.6 \\
    \midrule
    -0.1 & F & F & F & F & R & R \\
    -0.2 & F & R & R & R & R & R \\
    -0.3 & R & R & R & R & R & R \\
    -0.4 & S & S & R & R & R & R \\
    -0.5 & S & S & S & S & S & R \\
    -0.6 & S & S & R & R & R & R \\
    -0.7 & R & R & R & R & R & R \\
    -0.8 & F & R & R & R & R & R \\
    -0.9 & F & F & F & F & R & R \\
    \bottomrule
  \end{tabular}
  \caption{Table denoting the macroscopic region of the model for various points along the line $x = 1/4$. The model has fixed parameters $t=1/4$ and $\alpha =1/2$, while $\beta$ and $y$ are varied.}
  \label{tab:num-conj-2}
\end{table}

The asymptotics of the model simplify in the case $\beta = 0$ enough to prove the entire limit shape. In the two-periodic model, sending the two-periodic parameter to zero eliminates the rough region of the model creating a frozen-smooth boundary \cite{JM23}. By Proposition \ref{prop:right-side-asymptotics}, we know that the right side of the $t$-split model should exhibit this same behavior for $\beta = 0$. Thus, we only need to state what happens to the left of the interface when $\beta = 0$. To do so we define the functions
\begin{equation} \label{eq:f-tilde}
    \widetilde{F}_2(w) = \begin{pmatrix}
        0 & 0 \\
        \frac{1}{2}(w+1) & 1
    \end{pmatrix}
\end{equation}
and 
\begin{equation} \label{eq:c-tilde}
    \widetilde{c}(w) = \frac{2}{1+2\widetilde{g}_\alpha(w)},
\end{equation}
where
\begin{equation} 
    \widetilde{g}_{\alpha}(w) = \frac{2w+\alpha^2(w^2+1)}{4\sqrt{w(w+\alpha^2)(1+\alpha^2w)}}.
\end{equation}
When $\beta = 0$ we can restate the kernel to the left of the interface as:
\begin{proposition} \label{prop:b=0-case}
    For $m,m' < tN$ and $\beta = 0$, the $t$-split model has correlation kernel
    \begin{multline} \label{eq:simplified-kernel-left-b0}
        \left[\K_N(4m',2\xi'+j;4m,2\xi+i)\right]_{i,j=0,1} = \frac{\1_{m\leq m'}}{2\pi \ii} \oint_{\gamma_{0,1}}\frac{\dz}{z} z^{\xi'-\xi}\Phi_{\alpha}(z)^{m-m'} \\
        -\frac{1}{(2\pi\ii)^2}\oint_{\gamma_1}\dw\oint_{\gamma_{0,1}} \frac{\dz}{z(z-w)}\frac{w^{\xi'+N/2+tN}(z-1)^{2tN}}{z^{\xi+N/2+tN}(w-1)^{2tN}}r_{\alpha,2}(w)^{tN-m'}\\
        \times \widetilde{c}(w) F_{\alpha,2}(w) \widetilde{F}_{2}(w) \widetilde{F}_{2}(z) \Phi_{\alpha}(z)^{m-tN}.
    \end{multline}
\end{proposition}
This proposition is proven in Section \ref{section:asymptotics}. Applying the coordinate change $\tilde{\xi} = \xi + N/2 - tN$ and $\tilde{\xi}' = \xi' + N/2 - tN$ to the kernel in \eqref{eq:simplified-kernel-left-b0}, the resulting saddle function is exactly the saddle function of the two-periodic model, but scaled by $tN$ rather than $N$. As the simulations in Figure \ref{fig:b=0-example} suggest, when $\beta=0$, the limit shape to the left of the interface coincides with the limit shape of a rescaled two-periodic model, restricted to the left half. The only difference between the kernels is the inclusion of the functions $\widetilde{c}(w)$, $\widetilde{F}_2(w)$, and $\widetilde{F}_2(z)$, which have no impact on the steepest descent analysis. \\

\section{Derivation of the Correlation Kernel} \label{section:kernel-proof}

To derive the correlation kernel, we wish to use Theorem 3.1 from \cite{BD19}; however it cannot be immediately applied to this case because $\Phi_\eps(z)$ has a singularity on the unit circle. In Section 5.2 of \cite{BD19} the authors present a method for overcoming this issue by introducing a bias into the weights. We write the biased matrices as 
\begin{equation} \label{eq:phi-eps-a}
    \Phi_{\eps,a}(z) = \frac{1}{(1-a^2z^{-1})^2}
    \begin{pmatrix}1 & \eps^2 z^{-1}a^{-1} \\ \eps^{-2}a^{-1} & 1\end{pmatrix}
    \begin{pmatrix}1 & \eps^2 z^{-1}a \\ \eps^{-2}a & 1\end{pmatrix}
    \begin{pmatrix}1 & z^{-1}a^{-1} \\ a^{-1} & 1 \end{pmatrix}
    \begin{pmatrix}1 & z^{-1}a \\ a & 1 \end{pmatrix}
\end{equation}
where $0< a < 1$. Sending $a \to 1$ gives $\Phi_{\eps,a}(z) \to \Phi_\eps(z)$. Now we can apply Theorem 3.1 from \cite{BD19} to the biased $t$-split two-periodic model. We begin by explicitly defining the Wiener-Hopf factorization of the matrix 
\begin{equation}
    \phi_{N,a}(z) = \Phi_{\alpha,a}(z)^{tN}\Phi_{\beta,a}(z)^{(1-t)N}.
\end{equation}
We say that 
\begin{equation} \label{eq:WH-decomp}
    \phi_{N,a}(z) = \widetilde{\phi}_{-,N,a}(z) \widetilde{\phi}_{+,N,a}(z)
\end{equation}
is a Wiener-Hopf type factorization provided:
\begin{enumerate}
    \item $\tilde{\phi}_{+,N,a}^{\pm 1}$ is analytic for $|z| < 1$ and continuous for $|z| \leq 1$,
    \item $\tilde{\phi}_{-,N,a}^{\pm 1}$ is analytic for $|z|>1$ and continuous for $|z| \geq 1$,
    \item and $\widetilde{\phi}_{-,N,a} \to I_2z^{-N}$ as $z \to \infty$.
\end{enumerate}
The explicit form of the matrices $\widetilde{\phi}_{\pm,N,a}(z)$ is complicated for general $a$. However, the factorization is continuous in $a$ and the matrices are simple to write when $a =1$. See \cite[Section 4]{BD19} for more details on computing the factorization in the $2 \times 2$ matrix case. From \cite[Theorem 3.1]{BD19} we obtain the kernel, 
\begin{multline} \label{eq:thm-3.1}
    \Big[\mathbb{K}_{a,N}(4m',2\xi'+j;4m,2\xi+i)\Big]_{i,j=0}^1 = -\frac{\mathbb{I}_{m>m'}}{2 \pi \ii} \oint_{\gamma_{0,1}} \frac{\dz}{z} z^{\xi'-\xi} \phi_{m \to m', N, a}(z)\\
    + \frac{1}{(2\pi\ii)^2} \oint_{\gamma_a} \dw \oint_{\gamma_{0,1,a}} \frac{\dz}{z(z-w)} \frac{w^{\xi'}}{z^{\xi}} \phi_{m'\to N,N,a}(w) \widetilde{\phi}_{+,N,a}(w)^{-1} \widetilde{\phi}_{-,N,a}(z)^{-1}\phi_{0\to m,N,a}(z)
\end{multline}
where $\gamma_{a}$ is a simple, positively oriented contour that contains the points $a^2$ and $1$ and does not contain the points $0$ or $a^{-2}$ and $\gamma_{0,1,a}$ is a simple, positively oriented contour containing the points $0$, $1$, $a^2$, and $a^{-2}$ and surrounds the contour $\gamma_a$. The matrices $\phi_{m'\to m,N,a}$ are defined as in \eqref{eq:phi-m'-to-m}, with each constituent matrix replaced by its biased counterpart. 

We would like to take the limit $a \to 1$, however the $w$ integral contains singularities inside and outside of the curve $\gamma_a$ that converge to the same point. Our hope is to massage the integrand so that all the singularities lie inside of the contour, thus allowing us to take the required limit. We present a proof of this in cases: $0 < m' < tN$ and $tN < m' < N$. In both cases the proof involves the infinite repetition of matrix identities, eventually resulting in elimination of the poles.
\begin{remark}
    This proof is different than the proof presented in \cite{S25} and presents an alternative, simpler, method to proving the correlation kernel for that model. 
\end{remark}

\subsection{Case 1: Left of the Interface}
We will first present the proof for $0 < m' \leq tN$. Under this restriction we write, 
\begin{equation}
    \phi_{m'\to N,N, a}(w)\widetilde{\phi}_{+,N,a}(w)^{-1} = \Phi_{\alpha,a}(w)^{tN - m'}\Phi_{\beta,a}(w)^{(1-t)N}\widetilde{\phi}_{+,N,a}(w)^{-1}.
\end{equation}
Now apply \eqref{eq:WH-decomp} to get, 
\begin{equation}
    \Phi_{\alpha,a}(w)^{tN - m'}\Phi_{\beta,a}(w)^{(1-t)N}\widetilde{\phi}_{+,N,a}(w)^{-1} = \Phi_{\alpha,a}(w)^{-m'}\widetilde{\phi}_{-,N,a}(w).
\end{equation}
Insert the eigen-decomposition of $\Phi_{\alpha,a}$ into the above expression yields, 
\begin{equation}
    \Phi_{\alpha,a}(w)^{-m'}\widetilde{\phi}_{-,N,a}(w) = r_{\alpha,1,a}(w)^{-m'}F_{\alpha,1,a}(w)\widetilde{\phi}_{-,N,a}(w) + r_{\alpha,2,a}(w)^{-m'}F_{\alpha,2,a}(w)\widetilde{\phi}_{-,N,a}(w).
\end{equation}
The explicit definitions of $r_{\alpha,j,a}$ and $F_{\alpha,j,a}$ for $j=1,2$ are given in Appendix \ref{appendix:biased-info}. The functions $F_{\alpha,j,a}$ are analytic away from the negative real axis. The function $r_{\alpha,1,a}$ and $r_{\alpha,1,a}$ are both analytic and non-zero away from the negative real axis, except $r_{\alpha,1,a}$ has a pole of order 2 at $a^2$ and $r_{\alpha,2,a}$ has a zero of order 2 at $a^{-2}$.

Substituting the term 
\begin{equation}
    r_{\alpha,1,a}(w)^{-m'}F_{\alpha,1,a}(w)\widetilde{\phi}_{-,N,a}(w)
\end{equation}
back into \eqref{eq:thm-3.1} gives no poles at $a^{-2}$, and so the limit $a \to 1$ may be taken. The contribution in \eqref{eq:thm-3.1} from the term
\begin{equation} \label{eq:repeat-from-here}
    r_{\alpha,2,a}(w)^{-m'}F_{\alpha,2,a}(w)\widetilde{\phi}_{-,N,a}(w)
\end{equation}
still has poles both inside and outside of the contour. Applying \eqref{eq:WH-decomp} once more, we may write this as, 
\begin{equation}
    r_{\alpha,2,a}(w)^{-m'}F_{\alpha,2,a}(w)\widetilde{\phi}_{-,N,a}(w) = r_{\alpha,2,a}(w)^{-m'}F_{\alpha,2,a}(w)\Phi_{\alpha,a}(w)^{tN}\Phi_{\beta,a}(w)^{(1-t)N}\widetilde{\phi}_{+,N,a}(w)^{-1}.
\end{equation}
Recognize that, by orthogonality,
\begin{equation}
    F_{\alpha,2,a}(w)\Phi_{\alpha,a}(w)^{tN} = r_{\alpha,2,a}(w)^{tN}F_{\alpha,2,a}(w).
\end{equation}
Thus we obtain,
\begin{equation}
    r_{\alpha,2,a}(w)^{-m'}F_{\alpha,2,a}(w)\widetilde{\phi}_{-,N,a}(w) = r_{\alpha,2,a}(w)^{tN-m'}F_{\alpha,2,a}(w)\Phi_{\beta,a}(w)^{(1-t)N}\widetilde{\phi}_{+,N,a}(w)^{-1}.
\end{equation}
Next expand the above using the eigen-decomposition of $\Phi_{\beta,a}$: 
\begin{multline}
    r_{\alpha,2,a}(w)^{tN-m'}F_{\alpha,2,a}(w)\Phi_{\beta,a}(w)^{(1-t)N}\widetilde{\phi}_{+,N,a}(w)^{-1} \\
    = r_{\alpha,2,a}(w)^{tN-m'}r_{\beta,1,a}(w)^{(1-t)N}F_{\alpha,2,a}(w)F_{\beta,1,a}(w)\widetilde{\phi}_{+,N,a}(w)^{-1} \\
    + r_{\alpha,2,a}(w)^{tN-m'}r_{\beta,2,a}(w)^{(1-t)N}F_{\alpha,2,a}(w)F_{\beta,2,a}(w)\widetilde{\phi}_{+,N,a}(w)^{-1}.
\end{multline}
Substituting the second term on the right into \eqref{eq:thm-3.1} gives no poles inside $\gamma_a$, so the integral vanishes by the residue theorem. Thus we are left with the term 
\begin{equation}
    r_{\alpha,2,a}(w)^{tN-m'}r_{\beta,1,a}(w)^{(1-t)N}F_{\alpha,2,a}(w)F_{\beta,1,a}(w)\widetilde{\phi}_{+,N,a}(w)^{-1},
\end{equation}
which, when substituted back into \eqref{eq:thm-3.1}, still has poles both inside and outside of the contour $\gamma_a$. Applying \eqref{eq:WH-decomp} once more, 
\begin{multline}
    r_{\alpha,2,a}(w)^{tN-m'}r_{\beta,1,a}(w)^{(1-t)N}F_{\alpha,2,a}(w)F_{\beta,1,a}(w)\widetilde{\phi}_{+,N,a}(w)^{-1} \\
    = r_{\alpha,2,a}(w)^{tN-m'}r_{\beta,1,a}(w)^{(1-t)N}F_{\alpha,2,a}(w)F_{\beta,1,a}(w)\Phi_{\beta,a}(w)^{-(1-t)N}\Phi_{\alpha,a}(w)^{-tN}\widetilde{\phi}_{-,N,a}(w).
\end{multline}
Simplifying using orthogonality gives, 
\begin{multline}
    r_{\alpha,2,a}(w)^{tN-m'}r_{\beta,1,a}(w)^{(1-t)N}F_{\alpha,2,a}(w)F_{\beta,1,a}(w)\Phi_{\beta,a}(w)^{-(1-t)N}\Phi_{\alpha,a}(w)^{-tN}\widetilde{\phi}_{-,N,a} \\
    = r_{\alpha,2,a}(w)^{tN-m'}F_{\alpha,2,a}(w)F_{\beta,1,a}(w)\Phi_{\alpha,a}(w)^{-tN}\widetilde{\phi}_{-,N,a}(w).
\end{multline}
Expanding $\Phi_{\alpha,a}$ via its eigen-decomposition,
\begin{multline}
    r_{\alpha,2,a}(w)^{tN-m'}F_{\alpha,2,a}(w)F_{\beta,1,a}(w)\Phi_{\alpha,a}(w)^{-tN}\widetilde{\phi}_{-,N,a} \\
    = r_{\alpha,2,a}(w)^{-m'}F_{\alpha,2,a}(w)F_{\beta,1,a}(w)F_{\alpha,2,a}(w)\widetilde{\phi}_{-,N,a} \\
    + r_{\alpha,2,a}(w)^{tN-m'}r_{\alpha,1,a}(w)^{-tN}F_{\alpha,2,a}(w)F_{\beta,1,a}(w)F_{\alpha,1,a}(w)\widetilde{\phi}_{-,N,a}(w).
\end{multline}
Substituting the second term on the right into \eqref{eq:thm-3.1} yields no pole at $a^{-2}$, and the limit as $a \to 1$ may be taken. The first term on the right is not amenable to this limit, but it has the same structure as \eqref{eq:repeat-from-here}. The only difference is the additional $F$-matrices, which have no poles of concern around or inside the contour $\gamma_a$. We may therefore return to that step and iterate. Repeating the process $n$ times yields, 
\begin{multline}
    \phi_{m'\to N,N, a}(w)\widetilde{\phi}_{+,N,a}(w)^{-1} \equiv r_{\alpha,1,a}(w)^{-m'}F_{\alpha,1,a}(w)\widetilde{\phi}_{-,N,a}(w) \\
    + r_{\alpha,2,a}(w)^{tN-m'}r_{\alpha,1,a}(w)^{-tN} \left( \sum_{k=1}^n \left(F_{\alpha,2,a}(w)F_{\beta,1,a}(w)\right)^k \right) F_{\alpha,1,a}(w) \widetilde{\phi}_{-,N,a}(w) \\
    + r_{\alpha,2,a}(w)^{-m'}\left(F_{\alpha,2,a}(w)F_{\beta,1,a}(w)\right)^n F_{\alpha,2,a}(w)\widetilde{\phi}_{-,N,a}(w)
\end{multline}
where we use the equivalence symbol to denote that the value of the integral in equation \eqref{eq:thm-3.1} remains unchanged under the replacement. When the first two terms on the right hand side above are plugged into the integral we are able to immediately take the limit $a \to 1$. The third term is handled by the following lemma.
\begin{lemma} \label{lem:tech-lemma-1}
    As $n \to \infty$, $\left(F_{\alpha,2,a}(w)F_{\beta,1,a}(w)\right)^n \to 0$ uniformly for $w \in \gamma_a$. 
\end{lemma}
The proof is given in Section \ref{section:technical-lemmas}. Thus the last term vanishes as $n \to \infty$, and the limit $a \to 1$ may now be taken. In doing so, the integral in equation \eqref{eq:thm-3.1} becomes, 
\begin{multline} 
    \Big[\mathbb{K}_{a,N}(4m',2\xi'+j;4m,2\xi+i)\Big]_{i,j=0}^1 = -\frac{\mathbb{I}_{m>m'}}{2 \pi \ii} \oint_{\gamma_{0,1}} \frac{\dz}{z} z^{\xi'-\xi} \phi_{m \to m', N}(z)\\
    + \frac{1}{(2\pi\ii)^2} \oint_{\gamma_1} \dw \oint_{\gamma_{0,1}} \frac{\dz}{z(z-w)} \frac{w^{\xi'}}{z^{\xi}} r_{\alpha,1}(w)^{-m'}F_{\alpha,1}(w)\widetilde{\phi}_{-,N}(w) \widetilde{\phi}_{-,N}(z)^{-1}\phi_{0\to m,N}(z) \\
    + \frac{1}{(2\pi\ii)^2} \oint_{\gamma_1} \dw \oint_{\gamma_{0,1}} \frac{\dz}{z(z-w)} \frac{w^{\xi'}}{z^{\xi}} r_{\alpha,2}(w)^{2tN-m'}\left( \sum_{k=1}^\infty \left(F_{\alpha,2}(w)F_{\beta,1}(w)\right)^k \right) F_{\alpha,1}(w) \widetilde{\phi}_{-,N}(w) \\
    \times \widetilde{\phi}_{-,N}(z)^{-1}\phi_{0\to m,N}(z) 
\end{multline}
We are able to simplify the second integral above using the following lemma, 
\begin{lemma} \label{lem:tech-lemma-2}
    For any $w \in \gamma_1$,
    \begin{equation}
        \sum_{k=1}^\infty \left(F_{\alpha,2}(w)F_{\beta,1}(w)\right)^k = c(w)F_{\alpha,2}(w)F_{\beta,1}(w)
    \end{equation}
    where $c(w)$ is defined by equation \eqref{eq:c-def}.
\end{lemma}
The proof of this lemma is postponed to Section \ref{section:technical-lemmas}. It remains to record the explicit form of the Wiener–Hopf factorization. As presented in \cite[Theorem 5.2]{BD19} after taking the limit $a \to 1$, the factorization is simple. We have, 
\begin{equation}
    \widetilde{\phi}_{-,N}(w) = \frac{w^N}{(1-w)^N}\Phi_{\alpha}(w)^{tN}\Phi_{\beta}(w)^{(1/2 - t)N}.
\end{equation}

\begin{figure}
    \centering
    \begin{tikzpicture}[scale=1.2]
        \node at (0,5) {$r_{\alpha,1}^{-m'}F_{\alpha,1,a}\tilde{\phi}_{-,N,a} + r_{\alpha,2,a}^{-m'}F_{\alpha,2,a}\tilde{\phi}_{-,N,a}$};
        \node[rotate=270] at (-1.5,4.5) {$\to$};
        \node at (-1.5,4) {ready for limit};
        \node[rotate=90] at (1.5,4.5) {$\equiv$};
        \node at (2,4) {$r_{\alpha,2,a}^{tN-m'}F_{\alpha,2,a}\Phi_{\beta,a}^{(1-t)N} \tilde{\phi}_{+,N,a}^{-1}$};
        \node[rotate=90] at (1.5,3.5) {$\equiv$};
        \node at (1.5,3) {$r_{\alpha,2,a}^{tN-m'}r_{\beta,1,a}^{(1-t)N}F_{\alpha,2,a}F_{\beta,1,a} \tilde{\phi}_{+,N,a}^{-1}+r_{\alpha,2,a}^{tN-m'}r_{\beta,2,a}^{(1-t)N}F_{\alpha,2,a}F_{\beta,2,a} \tilde{\phi}_{+,N,a}^{-1}$};
        \node[rotate=90] at (4,2.5) {$\equiv$};
        \node at (4,2) {$0$};
        \node[rotate=90] at (0,2) {$\equiv\equiv\equiv\equiv$};
        \node at (0.5,1) {$r_{\alpha,2,a}^{-m'}F_{\alpha,2,a}F_{\beta,1,a}F_{\alpha,2,a}\tilde{\phi}_{-,N,a} + r_{\alpha,2,a}^{tN-m'}r_{\alpha,1,a}^{-tN}F_{\alpha,2,a}F_{\beta,1,a}F_{\alpha,1,a}\tilde{\phi}_{-,N,a}$};
        \node[rotate=270] at (2,0.5) {$\to$};
        \node at (2,0) {ready for limit};
        \node[rotate=270] at (-2,0.5) {$\to$};
        \node at (-2,0) {repeat the process!};
    \end{tikzpicture}
    \caption{Diagram to illustrate the kernel proof in the case of $m < tN$. We use equivalence symbols to denote that the expressions are interchangeable within the contour integral. }
    \label{fig:placeholder}
\end{figure}

\subsection{Case 2: Right of the Interface}
The proof for $m' > tN$ follows a structure identical to the preceding case, but with some subtle differences. Write,
\begin{equation}
    \phi_{m'\to N,N, a}(w)\widetilde{\phi}_{+,N,a}(w)^{-1} = \Phi_{\beta,a}(w)^{N-m'}\widetilde{\phi}_{+,N,a}(w)^{-1}.
\end{equation}
Expanding $\Phi_{\beta,a}$ via its eigen-decomposition gives,
\begin{equation}
    \phi_{m'\to N,N, a}(w)\widetilde{\phi}_{+,N,a}(w)^{-1} = \left(r_{\beta,1,a}(w)^{N-m'}F_{\beta,1,a}(w) + r_{\beta,2,a}(w)^{N-m'}F_{\beta,2,a}(w)\right)\widetilde{\phi}_{+,N,a}(w)^{-1}.
\end{equation}
Substituting the term
\begin{equation}
    r_{\beta,2,a}(w)^{N-m'}F_{\beta,2,a}(w)\widetilde{\phi}_{+,N,a}(w)^{-1}
\end{equation}
into \eqref{eq:thm-3.1} yields no poles inside $\gamma_a$, and so
\begin{equation}
    r_{\beta,2,a}(w)^{N-m'}F_{\beta,2,a}(w)\widetilde{\phi}_{+,N,a}(w)^{-1} \equiv 0
\end{equation}
with respect to the integral. The remaining term,
\begin{equation}
    r_{\beta,1,a}(w)^{N-m'}F_{\beta,1,a}(w)\widetilde{\phi}_{+,N,a}(w)^{-1},
\end{equation}
has poles both inside and outside the contour. Applying \eqref{eq:WH-decomp} gives,
\begin{equation}
    r_{\beta,1,a}(w)^{N-m'}F_{\beta,1,a}(w)\widetilde{\phi}_{+,N,a}(w)^{-1} = r_{\beta,1,a}(w)^{N-m'}F_{\beta,1,a}(w)\Phi_{\beta,a}(w)^{-(1-t)N}\Phi_{\alpha,a}(w)^{-tN}\widetilde{\phi}_{-,N,a}(w),
\end{equation}
and simplifying by orthogonality yields,
\begin{equation}
    r_{\beta,1,a}(w)^{N-m'}F_{\beta,1,a}(w)\widetilde{\phi}_{+,N,a}(w)^{-1} = r_{\beta,1,a}(w)^{tN-m'}F_{\beta,1,a}(w)\Phi_{\alpha,a}(w)^{-tN}\widetilde{\phi}_{-,N,a}(w).
\end{equation}
Expanding $\Phi_{\alpha,a}$ via its eigen-decomposition gives,
\begin{multline}
    r_{\beta,1,a}(w)^{N-m'}F_{\beta,1,a}(w)\widetilde{\phi}_{+,N,a}(w)^{-1} \\
    = r_{\beta,1,a}(w)^{tN-m'}r_{\alpha,1,a}(w)^{-tN}F_{\beta,1,a}(w)F_{\alpha,1,a}(w)\widetilde{\phi}_{-,N,a}(w) \\
    + r_{\beta,1,a}(w)^{tN-m'}r_{\alpha,2,a}(w)^{-tN}F_{\beta,1,a}(w)F_{\alpha,2,a}(w)\widetilde{\phi}_{-,N,a}(w).
\end{multline}
The first term on the right, when substituted into \eqref{eq:thm-3.1}, has no pole at $a^{-2}$ and is therefore amenable to the limit $a \to 1$. Applying \eqref{eq:WH-decomp} to the second term gives,
\begin{multline}
    r_{\beta,1,a}(w)^{tN-m'}r_{\alpha,2,a}(w)^{-tN}F_{\beta,1,a}(w)F_{\alpha,2,a}(w)\widetilde{\phi}_{-,N,a}(w) \\
    = r_{\beta,1,a}(w)^{N-m'}F_{\beta,1,a}(w)F_{\alpha,2,a}(w)F_{\beta,1,a}(w)\widetilde{\phi}_{+,N,a}(w)^{-1} \\
    + r_{\beta,1,a}(w)^{tN-m'}r_{\beta,2,a}(w)^{(1-t)N}F_{\beta,1,a}(w)F_{\alpha,2,a}(w)F_{\beta,2,a}(w)\widetilde{\phi}_{+,N,a}(w)^{-1}.
\end{multline}
The second term on the right has no poles inside $\gamma_a$ and so does not contribute to the contour integral. The expression is now in a form amenable to iteration. After $n$ iterations we obtain,
\begin{multline}
    \phi_{m'\to N,a,N}(w)\widetilde{\phi}_{+,N,a}(w)^{-1} \\
    \equiv r_{\beta,1,a}(w)^{tN-m'} r_{\alpha,1,a}(w)^{-tN} F_{\beta,1,a}(w)  \left( \sum_{k=0}^n \left(F_{\alpha,2,a}(w)F_{\beta,1,a}(w)\right)^k\right) F_{\alpha,1,a}(w)\widetilde{\phi}_{-,N,a}(w) \\
    + r_{\beta,1,a}(w)^{tN-m'} r_{\alpha,2,a}(w)^{-tN} F_{\beta,1,a}(w)\left(F_{\alpha,2,a}(w)F_{\beta,1,a}(w)\right)^n F_{\alpha,2,a}(w)\widetilde{\phi}_{-,N,a}(w).
\end{multline}
By Lemma \ref{lem:tech-lemma-1} the second term vanishes as $n \to \infty$. Taking the limit $a \to 1$ in the first term and applying Lemma \ref{lem:tech-lemma-2} then yields the result stated in Theorem \ref{theorem:kernel}.

\section{Asymptotics Proofs} \label{section:asymptotics}

We will start by proving Proposition \ref{prop:right-side-asymptotics}. As explained briefly in Section \ref{section:asymptoticresults}, it is sufficient to show $I_{\beta,N} \to 0$ as $N \to \infty$. The contour integral $I_{\beta,N}$ is stated explicitly in \eqref{eq:right-side-extra}. 

\begin{proof}
Recall that $|r_{\beta,1}(w)| > 1$ for all $w \in \gamma_1$ and let
\begin{equation}
    r = \min_{w\in\gamma_1} |r_{\beta,1}(w)|.
\end{equation}
Additionally, the matrix-valued function $c(w)F_{\alpha,1}(w)F_{\beta,2}(w)$ is analytic and bounded on $\gamma_1$. Let 
\begin{equation}
    M = \max_{w \in \gamma_1} \| c(w) F_{\alpha,1}(w) F_{\beta,2}(w) \|,
\end{equation}
where we use $\| \cdot \|$ to denote the entry-wise max. We then bound 
\begin{multline} \label{eq:right_side_ineq}
    \left|\left[I_{\beta,N}(4m',2\xi'+j;4m,2\xi+i)\right]_{i,j=0,1}\right| \leq r^{(2t-1)N} M \\
    \times \Bigg|\frac{1}{(2\pi\ii)^2}\oint_{\gamma_1}\dw\oint_{\gamma_{0,1}} \frac{\dz}{z(z-w)}\frac{w^{\xi'+N}(z-1)^N}{z^{\xi+N}(w-1)^N}r_{\beta,1}(w)^{N/2 - m'} F_{\beta,1}(w)\Phi_{\beta}(z)^{m-N/2} \Bigg|.
\end{multline}
The double contour integral that remains exactly appears in the correlation kernel of the two-periodic Aztec diamond. This integral is known to converge to a finite value, see \cite{CJ16, DK21}. Since $t < 1/2$, $r^{(2t-1)N} \to 0$ as $N \to \infty$, so $I_{\beta,N} \to 0$. 
\end{proof}

Next we derive the form of the kernel to the left of the interface when $\beta = 0$. This kernel is stated in Proposition \ref{prop:b=0-case}. We start from the form of the kernel presented in Lemma \ref{lem:simplified-left-kernel} and take the limit $\beta \to 0$.

\begin{proof}
    We start with the form of the kernel in \eqref{eq:kernel-left-simplified} and expand out $\Phi_\beta(z)$ using its eigen-decomposition
    \begin{multline} \label{eq:simplified-kernel-left-eq2}
        \left[\K_N(4m',2\xi'+j;4m,2\xi+i)\right]_{i,j=0,1} = \frac{\1_{m\leq m'}}{2\pi \ii} \oint_{\gamma_{0,1}}\frac{\dz}{z} z^{\xi'-\xi}\Phi_{\alpha}(z)^{m-m'} \\
        -\frac{1}{(2\pi\ii)^2}\oint_{\gamma_1}\dw\oint_{\gamma_{0,1}} \frac{\dz}{z(z-w)}\frac{w^{\xi'+N}(z-1)^N}{z^{\xi+N}(w-1)^N}r_{\alpha,2}(w)^{tN-m'}r_{\beta,2}(w)^{(1/2-t)N} r_{\beta,2}(z)^{(t-1/2)N}\\
        \times c(w) F_{\alpha,2}(w) F_{\beta,2}(w) F_{\beta,2}(z) \Phi_{\alpha}(z)^{m-tN} \\
        -\frac{1}{(2\pi\ii)^2}\oint_{\gamma_1}\dw\oint_{\gamma_{0,1}} \frac{\dz}{z(z-w)}\frac{w^{\xi'+N}(z-1)^N}{z^{\xi+N}(w-1)^N}r_{\alpha,2}(w)^{tN-m'}r_{\beta,2}(w)^{(1/2-t)N}r_{\beta,1}(z)^{(t-1/2)N} \\
        \times c(w) F_{\alpha,2}(w) F_{\beta,2}(w) F_{\beta,1}(z) \Phi_{\alpha}(z)^{m-tN}.
    \end{multline}
    It is necessary to observe that the functions $c$, $F_{\beta,j}$, and $r_{\beta,j}$ for $j \in \{1,2\}$ are analytic in $\beta$ on the contours of integration, which remain a positive distance from any singularities and branch cuts. The pointwise limits of $F_{\beta,2}$ and $c$ are given in \eqref{eq:f-tilde} and \eqref{eq:c-tilde}, respectively. Since 
    \[
    \lim_{\beta \to 0^+} r_{\beta,2}(w) = 0,
    \]
    the second double contour integral in \eqref{eq:simplified-kernel-left-eq2} is exactly 0 in the limit. We express the first double contour integral as
    \begin{multline}
        -\frac{1}{(2\pi\ii)^2}\oint_{\gamma_1}\dw\oint_{\gamma_{0,1}} \frac{\dz}{z(z-w)}\frac{w^{\xi'+N}(z-1)^N}{z^{\xi+N}(w-1)^N}r_{\alpha,2}(w)^{tN-m'}\left(r_{\beta,2}(w)r_{\beta,1}(z)\right)^{(1/2-t)N}\\
        \times c(w) F_{\alpha,2}(w) F_{\beta,2}(w) F_{\beta,2}(z) \Phi_{\alpha}(z)^{m-tN}.
    \end{multline}
    Observe that 
    \begin{equation}
        \lim_{\beta \to 0^+} r_{\beta,2}(w)r_{\beta,1}(z) = \frac{z(w-1)^2}{w(z-1)^2},
    \end{equation}
    so the desired results is obtained.
\end{proof}

\section{Technical Lemmas} \label{section:technical-lemmas}

We beginning by proving Lemma \ref{lem:simplified-left-kernel} stated in Section \ref{section:asymptoticresults}.

\begin{proof}[Proof of Lemma \ref{lem:simplified-left-kernel}]
    To begin we rewrite the kernel presented in equation \eqref{eq:kernel-left} under the assumption $m,\,m'<tN$,
    \begin{multline}
        \left[\K_N(4m',2\xi'+j;4m,2\xi+i)\right]_{i,j=0,1} = -\frac{\1_{m>m'}}{2\pi \ii} \oint_{\gamma_{0,1}}\frac{\dz}{z} z^{\xi'-\xi}\Phi_{\alpha}(z)^{m-m'}\\
        +\frac{1}{(2\pi\ii)^2}\oint_{\gamma_1} \dw \oint_{\gamma_{0,1}} \frac{\dz}{z(z-w)}\frac{w^{\xi'+N}(z-1)^N}{z^{\xi+N}(w-1)^N}r_{\alpha,1}(w)^{tN-m'}F_{\alpha,1}(w) \\
        \times \Phi_{\beta}(w)^{(1/2-t)N} \Phi_{\beta}(z)^{(t-1/2)N}\Phi_{\alpha}(z)^{m-tN} \\
        + \frac{1}{(2\pi\ii)^2}\oint_{\gamma_1}\dw\oint_{\gamma_{0,1}} \frac{\dz}{z(z-w)}\frac{w^{\xi'+N}(z-1)^N}{z^{\xi+N}(w-1)^N}r_{\alpha,2}(w)^{tN-m'} c(w) F_{\alpha,2}(w) F_{\beta,1}(w)F_{\alpha,1}(w) \\
        \times\Phi_{\beta}(w)^{(1/2-t)N} \Phi_{\beta}(z)^{(t-1/2)N} \Phi_{\alpha}(z)^{m-tN}.
    \end{multline}
    Our manipulation will be focus on the first double contour integral. First we deform the contour $\gamma_1$ so that it surrounds the branch cuts $(-\infty,-\alpha^{-2}]$ and $[-\alpha^2,0]$. We call the new contour $\gamma_{\text{Br}}$. This deformation is possible because there is no singularity at infinity. This manipulation produces an additional single integral term from the residue $z=w$. The single integral term has the form,
    \begin{equation}
        \frac{1}{2\pi\ii}\oint_{\gamma_{0,1}} \frac{\dz}{z} z^{\xi'-\xi} r_{\alpha,1}(z)^{m'-m} F_{\alpha,1}(z).
    \end{equation}
    At this point, we can pass the contour $\gamma_{\text{Br}}$ through the branch cuts and replace the functions $r_{\alpha,1}(w)$ and $F_{\alpha,1}(w)$ with their analytic continuations, $r_{\alpha,2}(w)$ and $F_{\alpha,2}(w)$. At this point the kernel has the form, 
    \begin{multline}
        \left[\K_N(4m',2\xi'+j;4m,2\xi+i)\right]_{i,j=0,1} \\
        = -\frac{\1_{m>m'}}{2\pi \ii} \oint_{\gamma_{0,1}}\frac{\dz}{z} z^{\xi'-\xi}\Phi_{\alpha}(z)^{m-m'} 
        + \frac{1}{2\pi\ii}\oint_{\gamma_{0,1}} \frac{\dz}{z} z^{\xi'-\xi} r_{\alpha,1}(z)^{m'-m} F_{\alpha,1}(z)\\
        - \frac{1}{(2\pi\ii)^2}\oint_{\gamma_{\text{Br}}} \dw \oint_{\gamma_{0,1}} \frac{\dz}{z(z-w)}\frac{w^{\xi'+N}(z-1)^N}{z^{\xi+N}(w-1)^N}r_{\alpha,2}(w)^{tN-m'}F_{\alpha,2}(w) \\
        \times \Phi_{\beta}(w)^{(1/2-t)N} \Phi_{\beta}(z)^{(t-1/2)N}\Phi_{\alpha}(z)^{m-tN} \\
        + \frac{1}{(2\pi\ii)^2}\oint_{\gamma_1}\dw\oint_{\gamma_{0,1}} \frac{\dz}{z(z-w)}\frac{w^{\xi'+N}(z-1)^N}{z^{\xi+N}(w-1)^N}r_{\alpha,2}(w)^{tN-m'} c(w) F_{\alpha,2}(w) F_{\beta,1}(w)F_{\alpha,1}(w) \\
        \times\Phi_{\beta}(w)^{(1/2-t)N} \Phi_{\beta}(z)^{(t-1/2)N} \Phi_{\alpha}(z)^{m-tN}.
    \end{multline}
    Notice that the sign on the first double contour integral changes because passing the contour through the branch cuts switches the direction of the contour. We can now combine the two double contour integrals. To do this, deform the contour $\gamma_{\text{Br}}$ back to the contour $\gamma_1$, which introduces another single integral term. We obtain, 
    \begin{multline}
        \left[\K_N(4m',2\xi'+j;4m,2\xi+i)\right]_{i,j=0,1} = -\frac{\1_{m>m'}}{2\pi \ii} \oint_{\gamma_{0,1}}\frac{\dz}{z} z^{\xi'-\xi}\Phi_{\alpha}(z)^{m-m'} \\
        + \frac{1}{2\pi\ii}\oint_{\gamma_{0,1}} \frac{\dz}{z} z^{\xi'-\xi} r_{\alpha,1}(z)^{m'-m} F_{\alpha,1}(z)
        + \frac{1}{2\pi\ii}\oint_{\gamma_{0,1}} \frac{\dz}{z} z^{\xi'-\xi} r_{\alpha,2}(z)^{m'-m} F_{\alpha,2}(z)\\
        + \frac{1}{(2\pi\ii)^2}\oint_{\gamma_1}\dw\oint_{\gamma_{0,1}} \frac{\dz}{z(z-w)}\frac{w^{\xi'+N}(z-1)^N}{z^{\xi+N}(w-1)^N}r_{\alpha,2}(w)^{tN-m'} \left(- F_{\alpha,2}(w) + c(w) F_{\alpha,2}(w) F_{\beta,1}(w)F_{\alpha,1}(w)\right) \\
        \times\Phi_{\beta}(w)^{(1/2-t)N} \Phi_{\beta}(z)^{(t-1/2)N} \Phi_{\alpha}(z)^{m-tN}.
    \end{multline}
    Recombining the single integral terms gives, 
    \begin{multline} \label{eq:kernel_manipulation}
        \left[\K_N(4m',2\xi'+j;4m,2\xi+i)\right]_{i,j=0,1} = \frac{\1_{m\leq m'}}{2\pi \ii} \oint_{\gamma_{0,1}}\frac{\dz}{z} z^{\xi'-\xi}\Phi_{\alpha}(z)^{m-m'} \\
        + \frac{1}{(2\pi\ii)^2}\oint_{\gamma_1}\dw\oint_{\gamma_{0,1}} \frac{\dz}{z(z-w)}\frac{w^{\xi'+N}(z-1)^N}{z^{\xi+N}(w-1)^N}r_{\alpha,2}(w)^{tN-m'} \left(- F_{\alpha,2}(w) + c(w) F_{\alpha,2}(w) F_{\beta,1}(w)F_{\alpha,1}(w)\right) \\
        \times\Phi_{\beta}(w)^{(1/2-t)N} \Phi_{\beta}(z)^{(t-1/2)N} \Phi_{\alpha}(z)^{m-tN}.
    \end{multline}
    Lastly, we will simplify the matrix expression in the double contour integral. We write,  
    \begin{multline}
        - F_{\alpha,2}(w) + c(w) F_{\alpha,2}(w) F_{\beta,1}(w)F_{\alpha,1}(w) = - F_{\alpha,2}(w) + c(w) F_{\alpha,2}(w) F_{\beta,1}(w)(I - F_{\alpha,2}(w)) \\
        = - F_{\alpha,2}(w) + c(w) F_{\alpha,2}(w) F_{\beta,1}(w) - c(w) F_{\alpha,2}(w) F_{\beta,1}(w)F_{\alpha,2}(w)
    \end{multline}
    We can compute, 
    \begin{align*}
        c(w)F_{\alpha,2}(w)F_{\beta,1}(w)F_{\alpha,2}(w) &= c(w)F_{\alpha,2}(w)(I-F_{\beta,2}(w))F_{\alpha,2}(w) \\
        &= c(w)F_{\alpha,2}(w)F_{\alpha,2}(w) - c(w)F_{\alpha,2}(w)F_{\beta,2}(w)F_{\alpha,2}(w) \\
        &= c(w)F_{\alpha,2}(w) - F_{\alpha,2}(w).
    \end{align*}
    Using this above gives, 
    \begin{align}
        - F_{\alpha,2}(w) + c(w) F_{\alpha,2}(w) F_{\beta,1}(w)F_{\alpha,1}(w) &= c(w) F_{\alpha,2}(w) F_{\beta,1}(w) - c(w) F_{\alpha,2}(w) \\
        &= c(w) F_{\alpha,2}(w) F_{\beta,1}(w) - c(w) F_{\alpha,2}(w) \\
        &= c(w) F_{\alpha,2}(w) ( F_{\beta,1}(w) - I) \\
        &=  - c(w) F_{\alpha,2}(w) F_{\beta,2}(w)
    \end{align}
    Plugging the above into equation \eqref{eq:kernel_manipulation} gives the statement presented in the lemma. 
\end{proof}
Next we would like to prove the two technical lemmas necessary for proving the explicit statement of the kernel. We start with Lemma \ref{lem:tech-lemma-1}.
\begin{proof}[Proof of Lemma \ref{lem:tech-lemma-1}]
    We can explicitly compute the eigenvalues of the matrix $F_{\alpha,2,a}(w)F_{\beta,1,a}(w)$. One eigenvalue is 0 while the other is, 
    \begin{equation}
        \lambda(w) = \frac{1}{2} - \frac{2(\alpha^2+\beta^2)(w+1)^2+w(\alpha^2-1)(\beta^2-1)(a+a^{-1})^2}{8\alpha\beta \sqrt{(w^2+x_\alpha w+1)(w^2+x_\beta w+1)}}.
    \end{equation}
    We must show that $|\lambda(w)| \leq \rho < 1$ for all $w \in \gamma_a$. Before we begin to bound the components of $\lambda(w)$, one should recognize that we are able to shrink the contour $\gamma_a$ down to a circle of radius $r$ around $a^2$, where $r$ can be made arbitrarily small. This is because the components of the contour integral in equation \eqref{eq:thm-3.1} have no poles away from the branch cuts besides at $a^2$ and $a^{-2}$. 

    Observe that $\lambda(w)$ is analytic in a neighborhood of $w = a^2$. This is because the factor $w^2+x_\eps w+1$, for $\eps = \alpha,\beta$, is positive and non-zero at $w = a^2$. Moreover, one can compute that 
    \begin{align}
        (w^2+x_\alpha w+1)(w^2+x_\beta w+1)\Big|_{w=a^2} &=(a^4+x_\alpha a^2+1)(a^4+x_\beta a^2+1) \\
        &=\frac{1}{16}(a^2+1)^4(\alpha+\alpha^{-1})^2(\beta+\beta^{-1})^2.
    \end{align}
    Plugging the above into $\lambda$ yields, 
    \begin{align}
        \lambda(a^2) &= \frac{1}{2} - \frac{2(\alpha^2+\beta^2)(a^2+1)^2+a^2(\alpha^2-1)(\beta^2-1)(a+a^{-1})^2}{2\alpha\beta(a^2+1)^2(\alpha+\alpha^{-1})(\beta+\beta^{-1})} \\[6pt]
        &= \frac{1}{2} - \frac{2(\alpha^2+\beta^2)(a^2+1)^2+(\alpha^2-1)(\beta^2-1)(a^2+1)^2}{2(a^2+1)^2(\alpha^2+1)(\beta^2+1)} \\[6pt]
        &= \frac{1}{2} - \frac{(\alpha^2+1)(\beta^2+1)(a^2+1)^2}{2(a^2+1)^2(\alpha^2+1)(\beta^2+1)} = 0.
    \end{align}
    Since $\lambda(a^2)=0$, $\lambda(w)$ is analytic in a neighborhood of $a^2$, and we can shrink the contour $\gamma_a$ down to an arbitrarily small circle of radius $r$, this implies we may choose $r$ small enough to guarantee that $\lambda(w) \leq \rho < 1$.
\end{proof}
Next we will prove Lemma \ref{lem:tech-lemma-2}. The above lemma was proven for any $a \in (0,1]$, however in the following lemma we restrict to $a = 1$ for simplicity as it occurs after we take the limit $a \to 1$.  
\begin{proof}[Proof of Lemma \ref{lem:tech-lemma-2}]
    Lemma \ref{lem:tech-lemma-1} is true for any $a \in (0,1]$, thus it is enough to justify that the sum in Lemma \ref{lem:tech-lemma-2} converges. We write,
    \begin{equation}
        \sum_{k=1}^\infty \left(F_{\alpha,2}(w)F_{\beta,1}(w)\right)^k = \left(I-F_{\alpha,2}(w)F_{\beta,1}(w)\right)^{-1} - I.
    \end{equation}
    To prove the above lemma, we only need to check that 
    \begin{equation}
        \left(I-F_{\alpha,2}(w)F_{\beta,1}(w)\right)^{-1} = I + c(w)F_{\alpha,2}(w)F_{\beta,1}(w).
    \end{equation}
    The above fact can be proven using the following relation,
    \begin{align*}
        c(w)F_{\alpha,2}(w)F_{\beta,1}(w)F_{\alpha,2}(w) &= c(w)F_{\alpha,2}(w)(I-F_{\beta,2}(w))F_{\alpha,2}(w) \\
        &= c(w)F_{\alpha,2}(w)F_{\alpha,2}(w) - c(w)F_{\alpha,2}(w)F_{\beta,2}(w)F_{\alpha,2}(w) \\
        &= c(w)F_{\alpha,2}(w) - F_{\alpha,2}(w).
    \end{align*}
    Now we can check
    \begin{multline*}
        \left(I -F_{\alpha,2}(w)F_{\beta,1}(w) \right)\left(I + c(w)F_{\alpha,2}(w)F_{\beta,1}(w)\right) \\
        = I + c(w)F_{\alpha,2}(w)F_{\beta,1}(w) - F_{\alpha,2}(w)F_{\beta,1}(w) - c(w)F_{\alpha,2}(w)F_{\beta,1}(w)F_{\alpha,2}(w)F_{\beta,1}(w) \\
        = I + c(w)F_{\alpha,2}(w)F_{\beta,1}(w) - F_{\alpha,2}(w)F_{\beta,1}(w) - (c(w)F_{\alpha,2}(w) - F_{\alpha,2}(w))F_{\beta,1}(w) 
        = I.
    \end{multline*}
\end{proof}

\bibliographystyle{abbrv}
\bibliography{bibliography}

\appendix
\section{Biased Matrices: Full Definitions and Properties} \label{appendix:biased-info}

The eigenvalues of $\Phi_{\eps,a}$, which is defined by equation \eqref{eq:phi-eps-a}, are
\begin{equation}
    r_{\eps,1,a}(z) = \frac{1}{(z-a^2)^2}\Big((z+1)^2+\frac{1}{2}(a+a^{-1})^2(\eps^2+\eps^{-2}) + (a+a^{-1})(\eps+\eps^{-1})\sqrt{z(z^2+x_\eps z+1)}\Big)
\end{equation}
and 
\begin{equation}
    r_{\eps,2,a}(z) = \frac{1}{(z-a^2)^2}\Big((z+1)^2+\frac{1}{2}(a+a^{-1})^2(\eps^2+\eps^{-2}) \\
    - (a+a^{-1})(\eps+\eps^{-1})\sqrt{z(z^2+x_\eps z+1)}\Big),
\end{equation}
where 
\begin{equation}
    x_\eps = \frac{1}{4}\left((a+a^{-1})^2(\eps^2+\eps^{-2})-2(a-a^{-1})^2\right).
\end{equation}
One should observe that $x_\eps \geq 2$, so we can take the branch cuts to be 
\begin{equation}
    \left( -\infty, -\frac{x_\eps}{2} - \sqrt{\frac{x_\eps^2}{4}-1}\right] \bigcup \left[ -\frac{x_\eps}{2} + \sqrt{\frac{x_\eps^2}{4}-1}, 0\right],
\end{equation}
where we assume $\sqrt{z(z^2+x_\eps z+1)}$ is positive for $z > 0$. Thus $r_{\eps,1,a}(z)$ is analytic and non-zero away from the cuts except at $z = a^2$, where there is a pole of order 2. Similarly, $r_{\eps,2,a}(z)$ is analytic and non-zero away from the cuts except at $z = a^{-2}$, where there is a zero of order 2. We also can define explicitly the following matrices, 
\begin{equation}
    F_{\eps,1,a}(z) = \begin{pmatrix}
        \frac{1}{2}-\frac{z(a+a^{-1})(\eps-\eps^{-1})}{4\sqrt{z(z^2+x_\eps z+1)}} & 
        -\frac{\eps(z+1)}{2\sqrt{z(z^2+x_\eps z+1)}} \\
        -\frac{z\eps^{-1}(z+1)}{2\sqrt{z(z^2+x_\eps z+1)}} & 
        \frac{1}{2}+\frac{z(a+a^{-1})(\eps-\eps^{-1})}{4\sqrt{z(z^2+x_\eps z+1)}} \\
    \end{pmatrix}
\end{equation}
and
\begin{equation}
    F_{\eps,2,a}(z) = \begin{pmatrix}
        \frac{1}{2}+\frac{z(a+a^{-1})(\eps-\eps^{-1})}{4\sqrt{z(z^2+x_\eps z+1)}}  & 
        \frac{\eps(z+1)}{2\sqrt{z(z^2+x_\eps z+1)}} \\
        \frac{z\eps^{-1}(z+1)}{2\sqrt{z(z^2+x_\eps z+1)}} &  
        \frac{1}{2}-\frac{z(a+a^{-1})(\eps-\eps^{-1})}{4\sqrt{z(z^2+x_\eps z+1)}}\\
    \end{pmatrix}.
\end{equation}

\end{document}